\documentclass[11pt] {amsart}

\usepackage{amsmath,amssymb, graphicx, amscd,latexsym,here,color}
\usepackage{verbatim}
\usepackage{amssymb}
\usepackage{amsbsy}
\usepackage{amscd}
\usepackage{amsmath}
\usepackage{amsthm}
\usepackage[mathscr]{eucal}
\makeatletter
\newtheorem{Theorem}{Theorem}
\newtheorem{MainTheorem}[Theorem]{Main Theorem}
\newtheorem{Lojasiewicz Join Theorem}[Theorem]{\L ojasiewicz Join Theorem}
\newtheorem{Lemma}[Theorem]{Lemma}
\newtheorem{Corollary}[Theorem]{Corollary}
\newtheorem{Proposition}[Theorem]{Proposition}

\newtheorem{Definition}[Theorem]{Definition}
\newtheorem{Remark}[Theorem]{Remark}

\newtheorem{Example}[Theorem]{Example}

\newtheorem{ Observation}[Theorem]{Observation}
\newtheorem{Assertion}[Theorem]{Assertion}

\newcommand\vphi{\varphi}

\newcommand\al{\alpha}

\newcommand\be{\beta}

\newcommand\Ga{\Gamma}

\newcommand\De{\Delta}

\newcommand\bfz{\mbox {\bf  z}}

\newcommand\nl{\newline}

\newcommand\LS{\rm{LS}\/}
\newcommand\wt{\rm{wt}\/}

\newcommand{{{\ord}}}{\rm{ord}\/}

\newcommand\Var{\rm{Var}\/}
\newcommand\Int{\rm{Int}\/}

\newcommand\Cone{\rm{Cone}\/}

\newcommand\inv{^{-1}}

\def\inv{^{-1}}

\begin{document}
\title[\L ojasiewicz exponents of a certain analytic  functions ]
{\L ojasiewicz exponents of  a certain analytic functions}

\author
[M. Oka ]
{Mutsuo Oka }
\address{Department of Mathematics,
{Tokyo  University of Science,}
{Kagurazaka 1-3, Shinjuku-ku,}
{Tokyo 162-8601}
}
\email { oka@rs.tus.ac.jp}

\begin{abstract}
We consider  the exponent of  \L ojasiewicz inequality  $\|\partial\,f(\mathbf z)\| \ge  c |f(\mathbf z|^\theta$ for two classes of analytic functions and we will give an explicit estimation for $\theta$.
First   we  consider   certain non-degenerate functions which is not  convenient.
In \S 3.4, we give an example of a polynomial for which   $\theta_0(f)$ is not constant on the moduli space and in  \S 3.5,  we show that the behaviors of  the \L ojasiewicz exponents  is not similar as  the Milnor numbers by an example.

 In the last section (\S 4), we give also an estimation for    product functions $f(\mathbf z)=f_1(\mathbf z)\cdots f_k(\mathbf z)$ associated to a  family of a certain convenient non-degenerate complete intersection varieties.  In either class, the singularity is   not isolated.
We will give  explicit estimations of the \L ojasiewicz exponent $\theta_0(f)$ using   combinatorial data of the Newton boundary of $f$.  We generalize this estimation for non-reduced function
$g=f_1^{m_1}\cdots f_k^{m_k}$. 
\end{abstract}
\maketitle

\maketitle

\section{Holomorphic functions
 and \L ojasiewicz exponents}
Consider a germ of an analytic function $f(\bfz)$ at the origin. There are two type of inequalities which are shown by S. \L ojasiewicz \cite{L1,L2}.  
\begin{eqnarray}
\|\partial f(\mathbf z)\|&\ge& c |f(\mathbf z)|^\theta,\, c\ne 0,  \,0 \le \exists\theta<1  \label{Loj-ineq},\,\forall\mathbf z\in U,\\
\|\partial f(\mathbf z)\|&\ge& c \|\mathbf z\|^\eta,\,c\ne 0,\, \exists\eta>0,\, \forall\mathbf z\in U
\end{eqnarray}
where $U$ is a sufficiently small neighborhood of the origin. 
Here $\partial f(z)$ is the gradient vector
$(\frac{\partial f}{\partial z_1},\dots,\frac{\partial f}{\partial z_n})$.
These equalities hold in a sufficiently small neighborhood of the origin.
For the second inequality, $f(\mathbf z)$ must have an isolated singularity at the origin.
In our previous paper \cite{O-Lo}, we considered the exponent of the second \L ojasiewicz inequality 
 for a non-degenerate  holomorphic function $f(\mathbf z)$ 
(or a mixed function  $f(\mathbf z,\bar{\mathbf z})$) with an isolated singularity at the origin.
For further information about  \L ojasiewicz inequality (2), we refer \cite{L1,L2,Ab1,Ab2, Br,B-K-O,Len,O-Lo,Oleksik,Oleksik2}. 

In this paper,  we are interested  in the exponent of the type (1)  for a certain type of holomorphic functions which may have non-isolated singularities at the origin. 
 This  inequality has been   originally studied by \L ojasiewicz    \cite{L1,L2} and then  
 by many other  authors.  Most of the researches  have been concentrated for the existence of the inequality in more general setting.
For example,  in the papers \cite{Loi,Kurd, K-P}, the authors show the existence of \L ojasiewicz inequality  for the o-minimal situation. 
We study the exponent of the inequality  (1) for a certain type of holomorphic functions which may have non-isolated singularities at the origin. 
They are either non-degenerate functions or  the product of convenient non-degenerate polynomials associated to a non-degenerate complete intersection variety.
The existence of \L ojasiewicz inequality is well-known for holomorphic functions (\cite{L1,L2}). 
We are interested in the vanishing  speed of the gradient vectors
$\partial f(\mathbf z)$ near the origin   in  comparison with that of the absolute 
value of $f$ when $\mathbf z$ goes to the origin. Thus
we are mostly   interested in the best possible $\theta$ which satisfies (1). This number is    the infinimum of  $\theta$'s which satisfy (1)  and we denote it by $\theta_0(f)$
 hereafter.

 Let $\mathbf z(t)$ be an analytic curve with $\mathbf z(0)=\mathbf 0$ and $\mathbf z(t)\in \mathbb C^n\setminus f\inv(0)$ for $t\ne 0$. 
Then  we compare the order of the both side of 
(1), after substituting $\mathbf z=\mathbf z(t)$ to  get the inequality:
\begin{eqnarray*}
 &{{{\ord}}}_t\|\partial f(\mathbf z(t))\|\le \theta\times  {{{\ord}}}_t\,{f(\mathbf z(t))}\\
 \end{eqnarray*}
 or equivalently 
 \begin{eqnarray*}\label{racio}
   \dfrac{{{{\ord}_t}}\,\|\partial f(\mathbf z(t))\|}{{{{\ord}_t}}f(\mathbf z(t))}\le \theta.\quad
 \end{eqnarray*}
Using the Curve Selection Lemma (\cite{Milnor,Hamm}), $ \theta_0(f)$  can be understood as 
 the supremum of the left side ratios  of the above inequality for all possible such curves $\mathbf z(t)$ and we call it  {\em the  \L ojasiewicz exponent} of the function $f$ for the \L ojasiewicz inequality
(1).

\section{Non-degenerate hypersurfaces} 
\subsection{Dual Newton diagram} We consider an  analytic function  (or a polynomial) 
\[
f(\mathbf z)=\sum_{\nu} c_\nu \mathbf z^\nu\]
 defined in the neighborhood of the origin. Recall that the Newton polyhedron  $\Ga_+(f)$ is the convex hull of 
the union $\bigcup_{\nu,c_\nu\ne 0} \,(\nu+\mathbb R_+^n)$.  
The Newton boundary $\Ga(f)$ is defined  by the union of compact faces of $\Ga_+(f)$.
 Let $N^+\subset \mathbb R^n$ be the space of non-negative weight vectors.
That is, a weight vector $P=(p_1,\dots,p_n)$ is in $ N^+$ if and only if $p_i\ge 0$. It defines linear function $\ell_P$ of the Newton polyhedron 
$\Ga_+(f)$ by  $\ell_P(\nu)=\sum_{i=1}^n p_i\nu_i$ for $\nu\in \Ga_+(f)$.
Its minimal value is denoted by $d(P,f)$
and the face of $\Ga_+(f)$, where this minimal value is taken,  is denoted as $\De(P,f)$. If no ambiguity is likely, we simply denote it as $d(P)$ and $\De(P)$. 
We recall  an equivalence relation in $N^+$ which gives a polyhedral conical structure  in $N^+$.
Two weight vectors $P,Q$ are {\em equivalent} if and only if $\De(P)=\De(Q)$ and this equivalence gives a conical 
polyhedral subdivision of $N^+$ which we call {\em the dual Newton diagram} and denote it as $\Ga^*(f)$.
The equivalence class of $P$ is denoted as $[P]$.

\subsection{Vanishing and non-vanishing weight vectors}
For $I\subset\{1,\dots, n\}$, we put $\mathbb C^I=\{\mathbf z\,|\, z_j=0,\forall j\not\in I\}$. Thus $\mathbb C^I\subset \mathbb C^n$.
The  subspace $\mathbf C^I$ is called {\em  a vanishing coordinate subspace} if $f^I\equiv 0$. 
Note that $f^I$ is the restriction of $f$ to $\mathbb C^I$.
Consider a weight vector $P=(p_1,\dots, p_n)$.  Put $I(P):=\{i\,|\, p_i=0\}$.  Assume that $I(P)\ne \emptyset$.
A weight vector $P$ is called {\em a vanishing weight vector} (respectively {\em non-vanishing weight vector} 
if  $d(P)>0$ (resp.  if  $d(P)=0$).  
Thus $\mathbb C^{I(P)}$ is a vanishing coordinate subspace if $P$ is a vanishing weight vector.
We denote the   sets of strictly positive weight vectors (i.e. $I(P)=\emptyset$), vanishing weight vectors and non-vanishing weight vectors  by  $\mathcal W_+(f),\,\mathcal W_v(f)$ and 
$\mathcal W_{nv}(f)$ respectively.  Hereafter we simply denote them as $\mathcal W_+,\,\mathcal W_v,\,\mathcal W_{nv}$, if no ambiguity is likely.
\subsection{Convexity of the equivalence  class}
Let $P$ be a weight vector in $N^+$ and let $[P]$ the set of equivalent  weight vectors. The equivalence class $[P]$ is the 
interior of a   polyhedral convex cone in $N^+$
and $\dim\,[P]=n-\dim\,\De(P)$. This follows from the obvious equality:
\[\begin{split}
\De(P)\cap\De(Q)&\ne \emptyset \implies\\
&\De((1-t)P+tQ)=\De(P)\cap\De(Q),\quad 0< t< 1.
\end{split}
\]
Put ${\LS}(P,Q):=\{(1-t)P+tQ\,|\, 0\le t\le 1\}$
and we call ${\LS}(P,Q)$ {\em the line segment with ends  $P,Q$}.
Consider the closure $\overline{[P]} $ of $[P]$ in the Euclidean topology.
Then $Q\in \overline{[P]}$ if and only if $\De(Q)\supset \De(P)$ and $\overline{[P]}$ is also a closed polyhedral convex cone.
We say $Q$ is on the boundary of $[P]$ if $[Q]\subset \overline{[P]}\setminus [P]$
 and denote as $Q\succ P$. 
 Note that $Q\succ P$ if and only if $\De(Q)\supsetneq \De(P)$. 
We visualize $\Ga^*(f)$ by cutting $N^+$ by some transversal  hyperplane $\Pi$ to the cone, say $\Pi:\,\nu_1+\dots+\nu_n=1$ and we see the  silhouette.  See Figure 2.
In the figure, the dimension of the equivalence class is one less.
Let $P$ a weight vector.  We say that $P$ is {\em a vertex of the dual Newton diagram  $\Ga^*(f)$}
if  and only if $\dim\,\De(P)=n-1$ or equivalently   $\dim\,[P]=1$.
\begin{Definition}{\rm
$f$ is called {\em $k$-convenient} if $f^I\not\equiv 0$ for any $I\subset \{1,\dots, n\}$ with $|I|\ge n-k$,  We say for simplicity $f$ is {\em convenient} if $f$ is 
$(n-1)$-convenient  (\cite{Okabook}). }
\end{Definition} 
\subsection{Face function, non-degeneracy and tameness} Let $\Xi$ be a face of $\Ga_+(f)$.
 The face function of  $\Xi$ is defined by
 $f_\Xi(\mathbf z):=\sum_{\nu\in \Xi}c_\nu\mathbf z^\nu$. For a weight vector $P$, we define 
$f_P(\mathbf z):=f_{\De(P)}(\mathbf z)$.  
If $P\in \mathcal W_+\cup\mathcal W_{v} $,  then $d(P)>0$ and $f_P$ is a weighted homogeneous polynomial of $\mathbf z_J$,  $J=I(P)^c$,  of degree $d(P)$ with respect to the weight $P$.  
Here $I(P)^c=\{1,\dots,n\}\setminus I(P)$ and 
$\mathbf z_I=(z_i\,|\, i\in I\}$. 
We recall that $f$ is {\em non-degenerate} if
 the mapping $f_P:\mathbb C^{*n}\to \mathbb C$ has no critical point  for any $P\in \mathcal W_+$.
Recall that  $\mathbb C^{*n}:=\{\mathbf z\in \mathbb C^n\,|\, z_i\ne 0,\,1\le i\le n\}$.

We say that  $f$ is  {\em locally tame} (or {\em strongly locally tame}) if for any  weight vector $P\in \mathcal W_v $,
the face function   $f_P: \mathbb C^{*I(P)^c}\to \mathbb C$ 
 has no critical point  as a function of variables $\mathbf z_{I(P)^c}$
  for any sufficiently small (resp. for any)
$\mathbf z_I\in \mathbb C^{*I(P)}$  fixed (\cite{EO14}). 
Here $\mathbb C^{*I}=\{\mathbf z\in \mathbb C^n\,|\, z_j=0, \,\text{iff}\, j\not\in I\}$. For $I=\{1,\dots, n\}$,
we write simply $\mathbb C^{*n}$ instead of $\mathbb C^{*I}$.
\begin{Definition}
{\rm
Let ${\Var}(P)=\{j\,|\, \frac{\partial f_P}{\partial z_j}\ne 0\}$. That is, $j\in \Var(P)$ if and only if $z_j$ appears in a monomial of $ f_P(\mathbf z)$. We call ${\Var}(P)$ {\em the variables} of $P$ or of  $f_P$.
Let $\widetilde I(P):=\bigcup\{I(Q)\,|,  Q\succ P, Q\in \mathcal W_{v}\}$ and we put $\widetilde {\Var}(P):={\Var}(P)\setminus \ \widetilde I(P)$. We call $ \{z_j\,|\, j\in\widetilde {\Var}(P)\}$ {\em invulnerable variables for $P$}.
Note that $\widetilde I(P)\supset \widetilde I(Q)$ if $Q\succ P$.  We introduce a stronger tameness:
$f$ is {\em strongly  inv-tame for $P$} if $\widetilde{\Var}(P)$ is not empty and   $f_P:\mathbb C^{*n}\to \mathbb C$ has no critical point as a polynomial of the invulnerable  variables $\{z_j,|\, j\in \widetilde{\Var}(P)\}$
for any $\mathbf z_{\widetilde I(P)}\in \mathbb C^{*\widetilde I(P)}$ fixed. 
We say $f$ is  {\em strongly  inv-tame} if  any weight vector $P\in \mathcal W_+\cup\mathcal W_v$ with $\dim\, \De(P)\ge 1$, $f$ is strongly  inv-tame for $P$.

For example, consider a weight vector $D$ on the open interval $\Int({RE_3})$ in $f_1(\mathbf z)$ (Figure 2). Then $R,E_3\succ D$. But $E_3\in \mathcal W_{nv}$. Thus $\widetilde I(D)=\{1\}$ and $f_{1D}=z_1^5z_2^2$ and we see $f_1$ is  strongly  inv-tame
for $D$.}
\end{Definition}
\begin{Remark}{\rm 
Assume that $f$ is $(n-2)$-convenient. Take  $P\in \mathcal W_+\cup\mathcal W_v$ with $\dim\, \De(P)\ge 1$. Assume that $Q\in \overline{[P]}$ and $Q\in \mathcal W_v$. Then $\sharp I(Q)=1$.  (If $\sharp(I(Q))=2$, $Q\in \mathcal W_{nv}$.)
If $P\in \mathcal W_v$, $I(Q)=I(P)$. Thus it is easy to check if $P$ is strongly   inv-tame or not.}
\end{Remark}

\subsection{Dimension of $[P]$} We recall the following relation of the dimension of the equivalence class $[P]$ and the dimension of $\De(P)$:
\[
\dim\,[P]=n-\dim\,\De(P). 
\]
Suppose that $I(P)\ne \emptyset$ and put $I:=I(P)$. Consider $f$ as a polynomial 
$f(\mathbf z)\in K[\mathbf z_{I^c}]$ with the coefficient ring $K:=\mathbb C[\mathbf z_{I}]$. We use the notation 
${}^Kf$ when we consider $f$  as a polynomial in $K[\mathbf z_{I^c}]$.  
Note  that $\dim\,\De_c(P)=\dim\,\De (P_{I^c},{}^Kf)$ and 
\[
\dim\,\De(P)=\dim\, \De_c(P)+\sharp I. 
\]
where $\sharp I$ is the cardinality of $I$. Here $\De_c(P)$ is defined by $\De(P)\cap \Ga(f)$.

\subsection{Normalized weight vector}
Take a weight vector $P=(p_1,\dots, p_n)\in \mathcal W_+\cup\mathcal W_v$. Then $d(P)>0$.
We consider the rational weight vector $\hat P=(\hat p_1,\dots, \hat p_n)$ which is defined by $\hat P= P/{d(P)} $. That is,  $\hat p_i=p_i/d(P),\,i=1,\dots,n$. It is clear that $P$ and $\hat P$ are equivalent and $d(\hat P)=1$.
We use this notation  throughout the paper and we call $\hat P$ {\em the normalized weight vector of $P$}.
If $d(P)=0$, $P$ does not have  any normalized  form.
Using the normalized weight vector, each monomial in $f_{\hat P}(\mathbf z)$ has weight $1$.
For given two weight vectors $P$ and $Q$ with $d(P),d(Q)>0$ and 
$\De(P,Q):=\De(P)\cap\De(Q)\ne \emptyset$, 
 consider the line segment ${\LS}(P,Q)=\{\hat P_t\,|\, 0\le t\le 1\}$ where
$\hat P_t:=(1-t)\hat P+t \hat Q$. 
 The  $i$-component $\hat p_{ti}$ of $\hat P_t$ is given as 
$\hat p_{ti}=(1-t)\hat p_i+ t\hat q_i$ and it is monotone  (either increasing or  constant or decreasing) function in $t$ for any $1\le i\le n$. In particular,
\begin{Proposition} There is  a canonical inequality:
$\hat p_{ti}\ge \min\{\hat p_i,\hat q_i\}$.
\end{Proposition}

\setcounter{figure}{0}
\begin{figure}[htb]
\setlength{\unitlength}{1bp}
\begin{picture}(600, 300)(-100,-20) 
{\includegraphics[width=8cm,  bb=0 0 595 842]{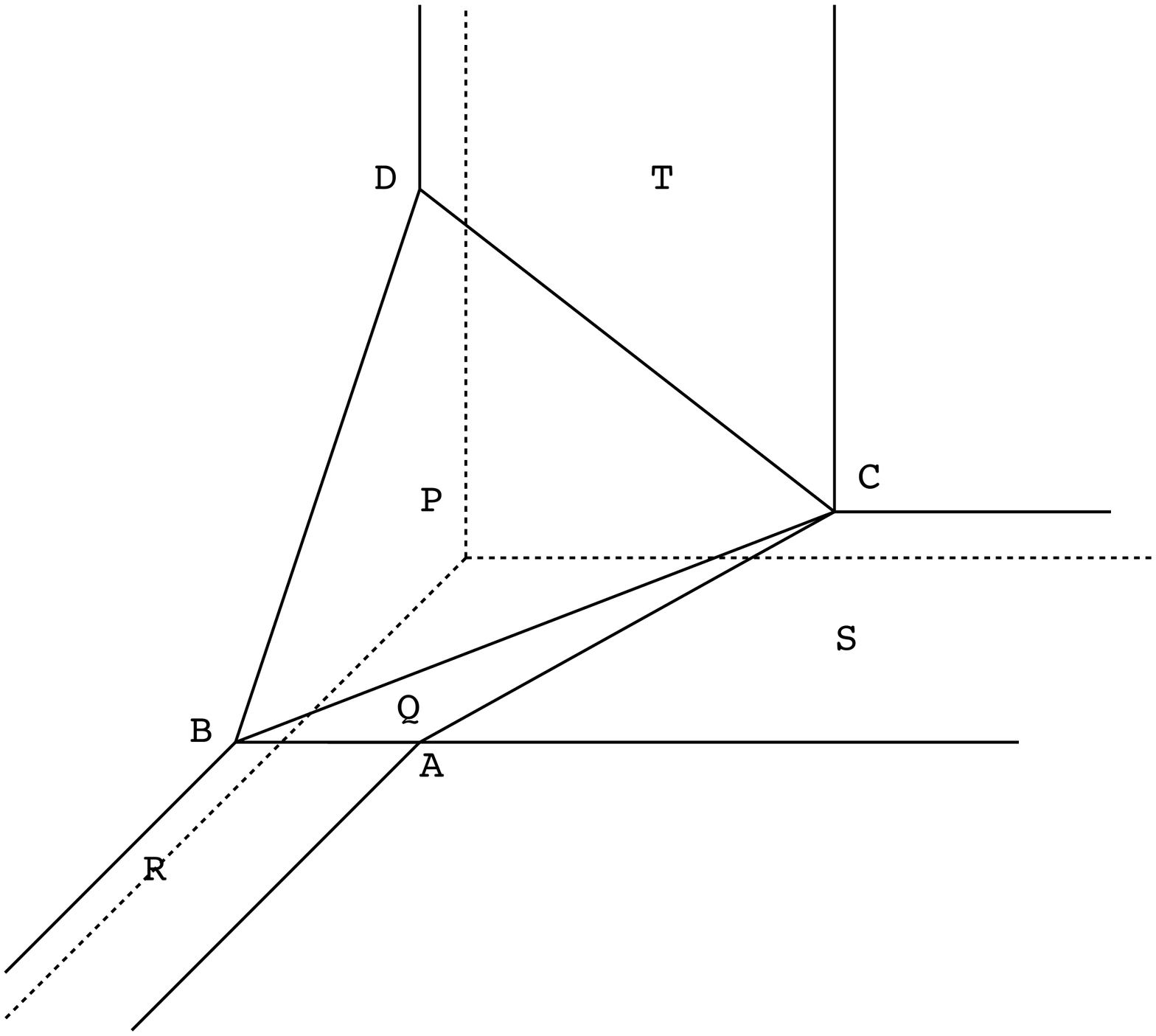}}
\put(-320,220){$f_1(\bfz)=z_1^5z_2^2+z_1^6z_3+z_2^6z_3^2+z_3^6z_1^3$}
\put(-300,200){$A=(5,2,0)\iff z_1^5z_2^2$}
\put(-300,185){$B=(6,0,1)\iff z_1^6z_3$}
\put(-300,170){$C=(0,6,2)\iff z_2^6z_3^2$}
\put(-300,155){$D=(3,0,6)\iff z_1^3z_3^6$}
\end{picture}
\vspace{-3cm}
\caption{Newton boundary of $f_1$}\label{NB1}
\end{figure}
\begin{figure}[htb]
\setlength{\unitlength}{1bp}
\begin{picture}(600, 300)(-100,-20) 
{\includegraphics[width=8cm,  bb=0 0 595 842]{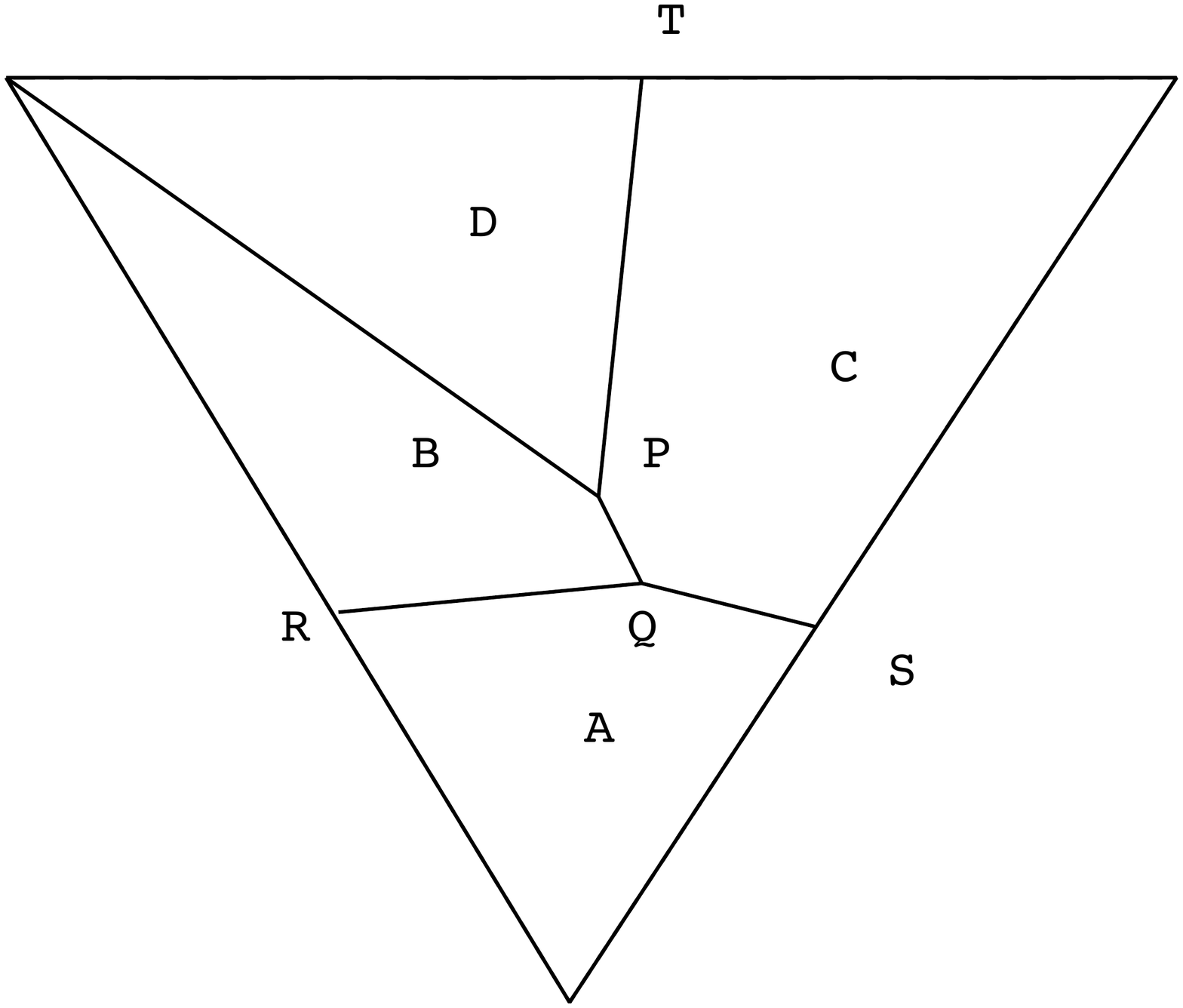}}
\put(-300,200){$E_1=(1,0,0)$}
\put(-300,185){$E_2=(0,1,0)$}
\put(-300,170){$E_3=(0,0,1)$}
\put(-106,238){$\bullet$}
\put(-115,160){$\bullet$}
\put(-108,143){$\bullet$}
\put(-73,133){$\bullet$}
\put(-165,137){$\bullet$}
\put(-227,238){$\bullet$}
\put(-0,238){$E_1$}
\put(-5,238){$\bullet$}
\put(-245,240){$E_2$}
\put(-133,60){$E_3$}
\put(-120,62){$\bullet$}
\put(-320,80){$f_1(\bfz)=z_1^5z_2^2+z_1^6z_3+z_2^6z_3^2+z_3^6z_1^3$}
\put(-300,155){$\hat P=(\frac 5{33},\frac 3{22},\frac 1{11})$}
\put(-300,140){$\hat Q=(\frac 4{27},\frac 7{54},\frac19)$}
\put(-300,125){$\hat R=(0,\frac12,1)$}
\put(-300,110){$\hat S=(\frac 15,0,\frac12)$}
\put(-300,95){$\hat T=(\frac 13,\frac16,0)$}
\end{picture}
\vspace{-3cm}
\caption{$\Ga^*(f_1)$}\label{DNB1}
\end{figure}


\begin{figure}[htb]
\setlength{\unitlength}{1bp}
\begin{picture}(600, 300)(-100,-20) 
{\includegraphics[width=8cm,  bb=0 0 595 842]{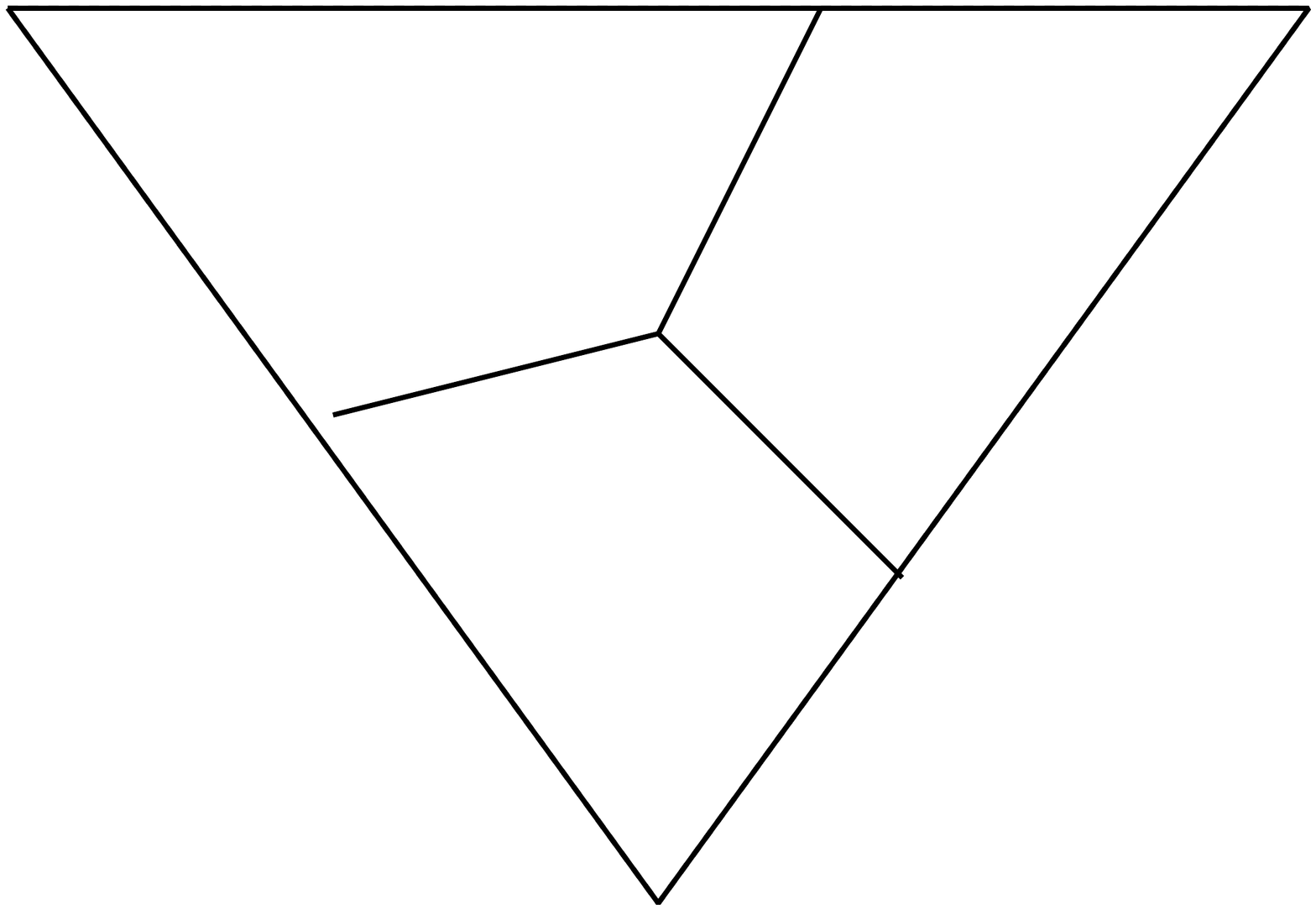}}

\put(0,230){$E_1$}
\put(-5,234){$\bullet$}
\put(-250,230){$E_2$}
\put(-230,235){$\bullet$}
\put(-100,80){$E_3$}
\put(-116,80){$\bullet$}
\put(-90,240){$T$}
\put(-88,235){$\bullet$}
\put(-104,178){$P$}
\put(-116,178){$\bullet$}
\put(-185,155){$R$}
\put(-175,163){$\bullet$}
\put(-67,141){$S$}
\put(-75,137){$\bullet$}
\put(-320,70){$\hat P=(\frac {4+bc-2c}{abc+8},\frac {4+ac-2a}{abc+8},\frac {4+ab-2b}{abc+8})$}
\put(-320,50){$\hat R=(0,\frac 1{2},\frac 1c),$}
\put(-250,50){$\hat S=(\frac1a,0,\frac 1{2}),$}
\put(-180,50){$\hat T=(\frac 1{2},\frac 1b,0)$}
\put(-320,30){$f_2=z_1^az_2^{2}+z_2^bz_3^{2}+z_3^cz_1^{2},\,a,b,c>2$}
\end{picture}
\vspace{-2.0cm}
\caption{ $\Ga^*(f_2)$}\label{DNB2}
\end{figure}

\begin{figure}[htb]
\setlength{\unitlength}{1bp}
\begin{picture}(600, 300)(-100,-20) 
{\includegraphics[width=8cm,  bb=0 0 595 842]{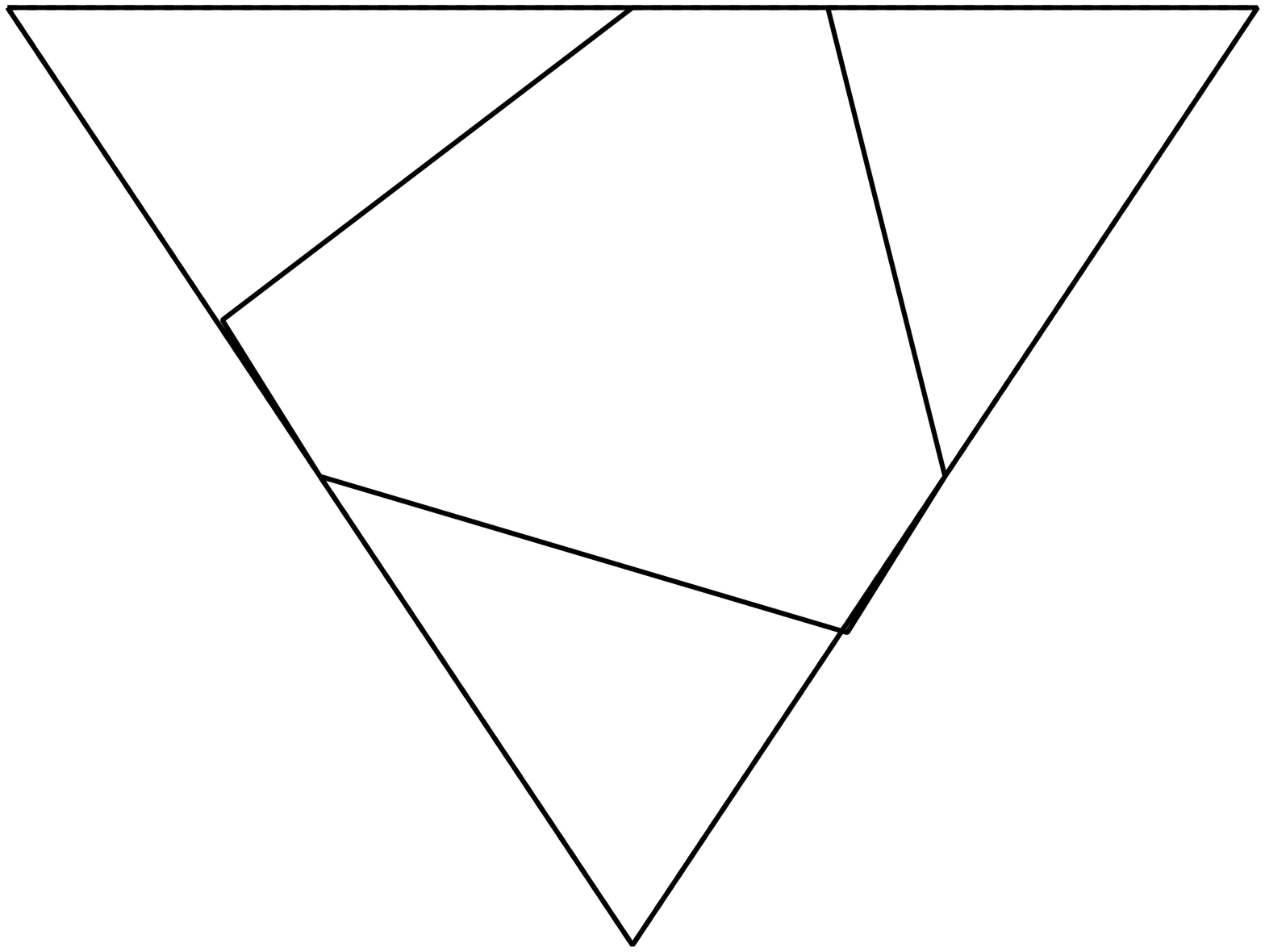}}

\put(0,240){$E_1$}
\put(-5,241){$\bullet$}
\put(-240,240){$E_2$}
\put(-226,241){$\bullet$}
\put(-100,80){$E_3$}
\put(-115,76){$\bullet$}
\put(-85,249){$R_2$}
\put(-81,242){$\bullet$}
\put(-115,249){$R_1$}
\put(-115,243){$\bullet$}
\put(-185,155){$S_1$}
\put(-170,157){$\bullet$}
\put(-200,185){$S_2$}
\put(-188,185){$\bullet$}
\put(-65,135){$T_2$}
\put(-77,132){$\bullet$}
\put(-55,150){$T_1$}
\put(-60,157){$\bullet$}
\put(-300,178){$\hat S_1=(0,\frac 12,\frac 12 )$}
\put(-300,163){$\hat S_2=(0,\frac 34,\frac 14)$}
\put(-300,148){$\hat T_1=(\frac 12,0,\frac 12)$}
\put(-300,133){$\hat T_2=(\frac 14,0,\frac 34)$}
\put(-300,118){$\hat R_1=(\frac12,\frac 12,0)$}
\put(-300,103){$\hat R_2=(\frac34,\frac 14,0)$}
\put(-300,88){$f_3=z_1^4z_2^2+z_2^4z_3^2+z_3^4z_1^2+z_1z_2z_3$}
\end{picture}
\vspace{-3.5cm}
\caption{ $\Ga^*(f_3)$}\label{DNB3}
\end{figure}

\subsection{Examples}
In the following,  polynomials in Example 1, Example 2 and Example  3 are all non-degenerate and strongly   inv-tame.
\begin{Example}\label{Example}
{\rm
Consider $f_1(\bfz)=z_1^4z_2^2+z_1^6z_3+z_2^6z_3^2+z_1^3z_3^6$. 
 Figure 1 and Figure 2 show the Newton boundary and the dual Newton diagram of $f_1$ respectively. In Figure 2, 
we see 8 equivalence  clasees which correspond to the vertices
$P,Q\in \mathcal W_+$, $R,S,T\in \mathcal W_v$ and $E_1,E_2,E_3\in \mathcal W_{nv}$, 11 equivalence  classes  corresponding to the edges and 4 equivalence  classes  which corresponds  to the interiors of four polyhedra in Figure 2. Here $E_1,E_2,E_3$ are the standard basis of $\mathbb R^3$.
}
\end{Example}

\begin{Example}
{\rm (Weighted homogeneous case) Consider the weighted homogeneous polynomial 
$f_2(\mathbf z)=z_1^az_2^{2}+z_2^bz_3^{2}+z_3^cz_1^{2}$  with $a,b,c>2$.
The dual Newton diagram is given in Figure 3.
}
\end{Example}
\begin{Example} {\rm Consider the polynomial
 $f_3(\mathbf z)=z_1^4 z_2^2+z_2^4 z_3^2+z_3^4z_1^2+z_1z_2z_3$.
This polynomial is  special  in the sense that its
 Newton boundary $\Ga(f)$ does not have any compact 2-dimensional face.
 See Figure 4.
 }
\end{Example}

\subsection{Ratio maps for curves in the coordinate subspaces}
Assume that $\mathbb C^I$ is a non-vanishing coordinate subspace.
Let $\mathcal  P_I$ be the set of  analytic curves $C:\mathbf z=\mathbf z(t),\, 0\le t\le 1$ such that  its image is 
in $\mathbb C^I\subset \mathbb C^n$, $\mathbf z(0)=\mathbf 0$ and 
$\mathbf  z(t)\in \mathbb C^{*I}\setminus V(f^I)$ for $t\ne 0$. 
For a  curve $\mathbf z(t)\in \mathcal P_I$, we consider the ratio map
\[
\theta : \mathcal P_I\to [0,1),\quad\theta(C)= \theta (\mathbf z(t))=\dfrac{{{{\ord}}}_t\|\partial f(\mathbf z(t))\|}{{{{\ord}}}_t f(\mathbf z(t))}.
\]
Note that $\|\partial f(\mathbf z(t))\|$ is measured in $\mathbb C^n$.
Consider a modified curve $\widetilde{\mathbf z}(t)$ defined as 
$\widetilde z_i(t)=z_i(t)$ for $i\in I$ and $=t^N$ for $i\not\in I$.  Note that   $\widetilde{\mathbf z}(t)\in \mathbb C^{*n}$ for $t\ne 0$.
Let  $\widetilde P=(\widetilde p_i,\dots,\widetilde p_n)$  be the weight vector of $\widetilde{\mathbf z}(t)$. 
Thus $\widetilde p_i=p_i$ for $i\in I$ and $\widetilde p_i=N$ for $i\not\in I$.
\begin{Proposition} \label{Reduced}
Taking 
$N$ sufficiently large we have that 
\[\begin{split}
&{\ord}_t\,\partial f(\widetilde{\mathbf z}(t))=
{\ord}_t\,\partial f(\mathbf z(t))\le {\ord}_t\, \partial f^I(\mathbf z(t)),\\
&{\ord}_t\, f(\mathbf z(t))={\ord}_t\,f(\widetilde{\mathbf z}(t)).
\end{split}
\]
Thus we have the equality:
\[
\frac{{\ord}_t\,\partial f(\mathbf z(t))}{{\ord}_t\,f(\mathbf z(t))}=\frac{{\ord}_t\,\partial f(\widetilde{\mathbf z}(t))}{{\ord}_t\,f(\widetilde{\mathbf z}(t))}.\]

\end{Proposition}
\begin{proof}
First observe that the difference 
\[\begin{split}
&f(\mathbf z)-f(\mathbf z_I), 
\,\,\frac{\partial f}{\partial z_i}(\mathbf z)-\frac{\partial f}{\partial z_i}(\mathbf z_I)\equiv 0\,\, 
\mod\,(z_j)_{j\not\in I}.
\end{split}
\]
Here $(z_j)_{j\not\in I}$ is the ideal generated by $z_j,\,j\not\in I$.
Therefore taking 
$N$ sufficiently large we may assume that 
   (i)
$\De(\widetilde P,f)=\De(P,f^I)$ 
and
\[{\rm (ii)}\quad
{\ord}_t \left (
\frac{\partial f}{\partial z_j}(\mathbf z(t))-\frac{\partial f}{\partial z_j}(\widetilde{\mathbf z}(t))\right
)\ge N
\]
and (iii) ${\ord}_t\,f(\mathbf z(t))={\ord}_t\,f(\widetilde{\mathbf z}(t))$. Then the assertion follows immediately.
\end{proof}
Thus  we have 
\begin{Corollary}\label{Reduced2}
$\theta(\mathbf z(t))=\theta(\widetilde{\mathbf z}(t))$.
\end{Corollary}
Therefore for the estimation of $\theta_0(f)$, it is enough to   consider  the case $I=\{1,\dots, n\}$, i.e. $\mathbf z(t)\in \mathbb C^{*n}$ for $t\ne 0$.
\begin{Lemma} \label{compact-vertex}
Consider a weight vector $P\in \mathcal W_+ $ and assume that 
 $f_P(\mathbf z)=c\mathbf z^\nu,\,c\ne 0$.  Consider a normalized weight $\hat P$. 
Put $\hat p_{max}:=\max\{\hat p_i\,|\, i\in \Var(P)\}$ and $|\nu|:=\sum_{i\in \Var(P)} \nu_i$.
Then 
 $\hat p_{max}\ge \frac{1}{|\nu|}$.  
 \end{Lemma}
\begin{proof}
The assertion is immediate from  the equality 
\[1=\sum_{i\in \Var(P)} \hat p_i\nu_i\le \hat p_{max}|\nu|.
\]
\end{proof}
\begin{Corollary}Consider  a curve $C$ parametrized  as in Lemma 10. 
Then 
\[
\theta(C)=1-\hat p_{max}\le 1-\frac 1{|\nu|}.
\]
\end{Corollary}
The assertion follows from the observation:  $\frac{\partial f_P}{\partial z_j}(\mathbf a)\ne 0$ for  any $j\in \Var(P)$
and  $\mathbf a\in \mathbb C^{*n}$.

\subsection{Ratio maps for  weight vectors} We assume that $f$ is strongly  inv-tame.
Let  $\mathcal W_+$ be the set of strictly positive weight vectors and let $\mathcal W_v$ be the subset of positive weight vectors such that   $I(P)\not=\emptyset$ and $d(P)>0$. 
Let $C=\{z=\mathbf z(t)\}\in \mathcal P$ and consider the Taylor expansion
\[
z_i(t)=a_it^{p_i}+\text{(higher terms)},\,\,a_i\ne 0,\,\,1\le i\le n.
\]
We consider the weight map 
$wt:\mathcal P\to \mathcal W$ by
${\wt}(\mathbf z(t))=P$ where $P=(p_1,\dots,p_n)$.
We want to estimate the ration
$\theta(\mathbf z(t))$ using  the weight vector ${\wt}(\mathbf z(t))$.
\begin{Definition}\rm{
 We  define the ratio maps for weight vector $P\in \mathcal W_+\cup\mathcal W_v$ as follows.
 First we put 
\[\begin{split}
&\hat p_{min}:=\min\{\hat p_j\,|\,  j\in \Var(P)\},\\
&\hat p_{min}':=\min\{\hat p_j\,|\,  j\in \widetilde{\Var}(P)\}.
\end{split}
\]
We define
\[\begin{split}
& \theta_i(P)=1-\hat p_i,\quad i\in \widetilde{\Var}(P),\,\dim\,\De{(P)}\ge 1,\\
&\theta(P)=1-\hat p_{min},\\
&\theta(P)'= 
\begin{cases}1-{\hat p_{min}'}, \quad &\dim\,\De{(P)}\ge 1,\\
1-\frac 1 {|\nu|},\,\quad &\dim\,\De(P)=0,\,f_P(\mathbf z)=c_\nu \mathbf z^\nu\end{cases}.
\end{split}
\]
Here  $P=(p_1,\dots, p_n)$ and $\hat P=(\hat p_1,\dots, \hat p_n)$ is the normalized weight vector of $P$.
}
\end{Definition}
As a special case, we have
\begin{Proposition}
Assume that $\dim\, \De(P)=n-1$. Then
  $I(P)=\widetilde {I}(P)$ and  $\hat p_{min}'=\min\,\{\hat p_j\,|\, j\in \Var(P)\setminus I(P)\}$.
  In particular, $\hat p_{min}'=\hat p_{min}$ if $P\in \mathcal W_+$.
\end{Proposition}


\subsection{Admissible line segments} We assume that $f$ is strongly inv-tame and non-degenerate.
Consider a weight vector $R\in \mathcal W_+\cup\mathcal W_v$ with $\dim\,\De (R) \ge 1$.
  Consider a line segment ${\LS}(P,Q)$ passing through $R$ with two weight vectors 
$P,Q$ on the boundary of $[R]$.
Recall that ${\LS}(P,Q)=\{P_s=(1-s)  P+s Q\,|\,0\le s\le 1\}$. By the assumption,   $ R=P_{s_0}, \,0<\exists s_0<1$.
We devide the situation into three cases  depending the end points $P,Q$.
\subsubsection{Strictly positive line segment}
Let $R$ be as above.
We say that the boundary of  $[R]$ is {\em strictly positive } if 
 the closure of the equivalence class $\overline{[R]}$  contains only strictly positive weight vectors. 
Thus  $P,Q\in \mathcal W_+$.  We use the line segment expression using the normalized vectors
$\hat P_s=(1-s)\hat P+s \hat Q$.  Then $R=P_{s_0'}$ for some $s_0'$. As the normalized weight  vector $\hat P_s$ is given as $\hat p_{sj}=(1-s)\hat p_j+s\hat q_j$ is a monotone linear function in $s$,
 it is easy to see that 
\begin{eqnarray}
&\theta_j(\hat R)&\le \max\{\theta_j(\hat P),\theta_j(\hat Q)\},\quad j\in {\Var}(\hat R),\\
&\theta(\hat R)'&\le \max\{\theta(\hat P),\theta(\hat Q)\}. 
\end{eqnarray}
Note that  in this case,  $\widetilde I(R)=\emptyset$ and  $\theta(\hat R)'=\theta(\hat R)$, $\theta(\hat P)'=\theta(\hat P )$ and $\theta(\hat Q)'=\theta(\hat Q)$,
as $\widetilde I(R)=\emptyset$.
For example, take $R$ on the line segment ${\LS}(P,Q)$ in Example 1 ( Figure 2).
\begin{Remark}
Note that ${\Cone}({\LS}(P,Q))={\Cone}({\LS}(\hat P,\hat Q))$,  though $(1-t)\hat P+t\hat Q$ is not necessarily the normalized vector of $(1-t)P+tQ$. Here for a subset $K\subset N^+$, we put ${\Cone}( K):=\{r\,P\,|\, P\in K,\,r>0\}$.
\end{Remark}

\subsubsection{Vanishing line segment}\label{vanishing line1}
 We  say that
${\LS}(P,Q)$ is {\em a vanishing line segment} if  $P, Q$ are in $\mathcal W_+\cup\mathcal W_v$ and  at least
one of $P$ or $Q$ is in $\mathcal W_v$.
\begin{Lemma} \label{VL1} Assume that ${\LS}(P,Q)$ is a vanishing line segment.
We assume that  $P\in \mathcal W_v$,    $Q\in \mathcal W_+\cup\mathcal W_v$.
 Consider the family of the  normalized weight vectors $\hat P_s$ for  the line segment ${\LS}(P,Q)$ which is defined as 
$\hat P_s=(1-s)\hat P+s\hat Q$  for 
$0\le  s\le1 $ and $\hat R=\hat P_{s_0'}$ ($0<\exists s_0'<1$). Then
 \[\begin{split}
& \theta_j(\hat R)\le \max\{\theta_j(\hat P),\theta_j(\hat Q)\},  \quad j\in\widetilde{\Var}(R),\\
&\theta(\hat R)'\le \max\{\theta(\hat P)',\theta(\hat Q)'\}.
\end{split}
\]
\end{Lemma}
\begin{proof}
 By the strong   inv-tameness, 
$\widetilde{\Var}(R)\ne \emptyset$ and   there exists a $j\in \widetilde{\Var}(R)$ so that 
$\frac{\partial f_{\hat P_{s_0'}}}{\partial z_j}(\mathbf a)\ne 0$. This implies $\theta(\hat R)'\le 1-\hat p_{s_0', j}$.
The assertion follows from the monotonicity  of $\hat p_{tj}$.
\end{proof}
\subsubsection{Non-vanishing line segment}\label{non-vanishing line} A line segment
${\LS}(P,Q)$ is called {\em non-vanishing line segment } if one of $P,Q$ is a non-vanishing weight vector.
Assume that  $P\in \mathcal W_{nv}$ and $Q\in \mathcal W_+\cup \mathcal W_v$ so that $\De(Q)\cap\De(P)\supset \De(R)$. 
Recall that 
${\LS}(P,Q)$ is defined  by $\{P_s\,|\, 0\le s\le 1\}$ where $P_s:=(1-s)P+sQ$ and $R=(1-s_0)P+s_0 Q$ as before.
The normalized weight  vectors of this  family  is written as $ \hat P_\tau:=\tau P+\hat Q$ with $\tau=\frac{1-s}s$ for $s\ne 0$.
In this parameter $\tau$, $0\le \tau<\infty$ and $\hat R=\hat P_{\tau_0},\,\exists \tau_0>0$.
Note that $\hat p_{\tau,j}$ is monotone increasing (or constant) in $\tau$ for any $j$. 
That is $\hat p_{\tau,j}\ge \hat p_{0,j}=\hat q_j$.
\begin{Lemma}\label{NVL}
We have the inequality:
\[\begin{split}
&\theta_{j}(\hat R)\le\theta_{j}(\hat Q),\, \quad j\in\widetilde {\Var}(R),\\
&\theta(\hat R)'\le \theta(\hat Q)' .\end{split}
\]
\end{Lemma}
\begin{Remark}\label{non-existence}{\rm
We  do not need  to consider the case where $P,Q$ are both non-vanishing, as $R$ is assumed to be 
in $\mathcal W_+\cup\mathcal W_v$. In the inductive argument on $\dim\, [R]$,  if the line segment is as in Lemma 16,
we  continue to work only for $Q$.
}
\end{Remark}

\section{Main result on non-degenerate functions}
\subsection{Convenient case}\label{equality}\label{Convenient-hypersurface}
Assume that $f(\mathbf z)=\sum_{\nu} {c_\nu}\mathbf z^\nu$ is a convenient non-degenerate analytic function 
in the sense of Kouchnirenko \cite{Ko}
and let $b_j$ be the point $\Ga(f)\cap\{$j$-\text{th coordinate axis}\}$.  Then $V(f)$ has an isolated singularity at the origin (Theorem (3.4),\cite{Okabook}, Corollary 20, \cite{OkaMix}).
Consider an analytic curve $\mathbf z(t),\,0\le t\le 1$ as in the previous section. Namely $\mathbf z(0)=\mathbf 0$ and 
$\mathbf z(t)\in \mathbb C^n\setminus f\inv(0)$ for $t\ne 0$.
We first assume that $\mathbf z(t)\in \mathbb C^{*n}$ for $t\ne 0$. 
Assume that $\mathbf z(t)$ has the following Taylor expansion:
\begin{eqnarray}\label{Taylor1}
&z_j(t)&=a_jt^{p_j}+\text{(higher terms)}, \quad a_j\ne 0,\, j=1,\dots,n,\\
&\frac{\partial f}{\partial z_j}(\mathbf z(t))&=\frac{\partial f_P}{\partial z_j}(\mathbf a) t^{d(P)-p_j}+\text{(higher terms)},\notag\\
&f(\mathbf z(t))&=b t^{d'}+\text{(higher terms)}, \, b\ne 0,\,d'\ge d(P).\notag
\end{eqnarray}
where 
 $P=(p_1,\dots, p_n)$, $\mathbf a=(a_1,\dots,a_n)$.  We use the same notation as in \cite{Okabook}. 
 Put  $p_{ min}=\min\{p_j\,|\, j\in \Var(P)\}$.
 Choose  index $ 1\le \al\le n$ so that $p_\al= p_{min}$.  Note that  $ \al$ may not unique but we fix it.
If $f(\mathbf z)$ is non-degenerate, then there exists $j_0$ so that $\frac{\partial f_P}{\partial z_{j_0}}(\mathbf a)\ne 0$. Thus
\begin{eqnarray}
& {{{{\ord}_t}}}\, \partial f(\mathbf z(t))\le d(P)-p_{j_0}\le d(P)-p_{min} ,\\
& \dfrac{{{{{\ord}_t}}}\,\partial f(\mathbf z(t))}{{{{\ord}_t}}\,f(\mathbf z(t))}\le \dfrac{d(P)-p_{min}}{d'}\le \dfrac{d(P)-p_{min}}{d(P)}.\label{eq7}
\end{eqnarray}
Use the normalized weight $\hat P=(\hat p_1,\dots, \hat p_n)$ where $\hat p_j:=p_j/d(P)$.
The right side of (7) is equal to $1- 1/{b_\al'}$ where $b_\al'$ is the $\al$-coordinate  so that $\hat p_\al b_\al'=1$.
This may not be an integer but we know that 
 $b_\al'\le b_\al$.
Thus we obtain
\[ \frac{{{{{\ord}_t}}}\,\partial f(\mathbf z(t))}{{{{\ord}_t}}\,f(\mathbf z(t))}\le 1-\frac 1{b_\al},\quad\forall \al,\,p_\al=p_{min}.
\]
Put $B:=\max\{b_j\,|\, j=1,\dots, n\}$ and let $I_B=\{j\,|\, b_j=B\}$. 
Thus the above estimation gives also 
$\theta(\mathbf z(t))= \frac{{{{{\ord}_t}}}\,\partial f(\mathbf z(t))}{{{{\ord}_t}}\,f(\mathbf z(t))}\le 1- 1/{B}$.
\begin{Definition}{\rm
The monomial $z_j^{b_j}$  is called  {\em \L ojasiewicz monomial} if $b_j=B$, i.e. $j\in I_B$. The monomial
$z_j^{b_j}$
is called 
{\em \L ojasiewicz exceptional } if $j\in I_B$ and 
  there exists $k\ne j$ and  a monomial $z_j^{B'}z_k$  in $f(\mathbf z)$ with $B'<B-1$. Otherwise $z_j^B$ is  called {\em a 
  non-exceptional \L ojasiewicz  monomial} (\cite{O-Lo}). }
\end{Definition}
\begin{Proposition}
If $f$ has a non-exceptional  \L ojasiewicz monomial, there exists an analytic curve $\mathbf z(t)$ so that 
the equality, 
$\theta(\mathbf z(t))=1- 1/B$, holds and thus 
\[
\theta_0(f)=1-\frac 1B.
\]
\end{Proposition}
\begin{proof}
To see this,  assume that  $z_{j_0}^{b_{j_0}}$ is a non-exceptional \L ojasiewicz monomial. 
Consider   the analytic curve $\mathbf z(t)$ which is defined by  $z_{j_0}(t)=t$ and $z_j(t)=t^N$ for any $j\ne j_0$ where $N$ is a sufficiently large positive integer.
Then it is easy to see that $\frac{\partial f}{\partial z_{j_0} }(\mathbf z(t))=c t^{B-1}+\text{(higher terms)},\,c\ne 0$.
If the derivative $\frac{\partial f}{\partial z_{j} }(\mathbf z),\,j\ne j_0$ contains a monomial $z_{j_0}^a$, it comes from the monomial $z_jz_{j_0}^a$ in $f(\mathbf z)$. By the assumption, $a\ge B-1$. Thus 
 ${\ord}_t\, \frac{\partial f}{\partial z_{j} }(\mathbf z(t))\ge B-1$  for $j\ne j_0$ and  $f(\mathbf z(t))=c t^{B}, c\ne 0$. Therefore it is easy to see that
the equality is satisfied. \end{proof}
\begin{Theorem} Assume that $f(\mathbf z)$ is a convenient non-degenerate function.
Then $\theta_0(f)\le 1- 1/B$.  Furthermore if  there exists a non-exceptional \L ojasiewicz monomial, the equality holds.
\end{Theorem}
\begin{Example}{\rm
Consider $f(\mathbf z)=z_1^5+z_1^3z_2+z_2^4+z_3^4$. Then $B=5$ but $z_1^5$  is an exceptional  \L ojasiewicz  monomial. In fact,  \L ojasiewicz exponent is given by 
$\theta_0(f)=1-\frac 14=3/4$.
}
\end{Example}
\subsection{Non-convenient case}
We assume that $f(\mathbf z)$ is non-degenerate and strongly inv-tame but we do not assume the convenience of  $f$.
The singularity is not necessarily isolated.
 Let $\mathcal D$ be the set of equivalent classes $[P]$ with $\dim\,[P]=n$.
 
 For a $[P]\in \mathcal D$, the face function is given as a monomial function $f_P(\mathbf z)=c\,\mathbf z^{\nu(P)}$ and we associate  the total degree
 $|\nu(P)|$ to $[P]$.
\begin{MainTheorem}\label{main-theorem1}
 Assume that $f$ is  non-degenerate and  strongly   inv-tame  function.
Then  the \L ojasiewicz exponent of type (1) has the estimation:
\[
\theta_0(f)\le \max\left\{L,\, \widetilde\theta\right\}
\]
where 
\[\begin{split}
&\widetilde\theta:=\max\,\{1-\hat p_{min}'\,|\, \hat P\in \mathcal W_+\cup\mathcal W_v,\, \dim\,\De(P)=n-1\},\\
&L=\max\{1-{1}/{|\nu(P)|}\,|\,  [P]\in \mathcal D\}.
\end{split}
\]
\end{MainTheorem}

\begin{proof}
Consider an analytic curve $C\in \mathcal P$ defined by $\mathbf z=\mathbf z(t),\,0\le t\le 1$ and $f(\mathbf z(t))\ne 0$ for $t\ne 0$. We may assume that $\mathbf z(t)\in \mathbb C^{*n}$ for $t\ne 0$ by Proposition 8.
 Assume that $\mathbf z(t)$ has the following expansion:
\begin{eqnarray}\label{Taylor2}
z_j(t)&=&a_jt^{r_j}+\text{(higher terms)}, \quad a_j\ne 0,\, j=1,\dots,n,\\
\frac{\partial f}{\partial z_j}(\mathbf z(t))&=&\frac{\partial f_R}{\partial z_j}(\mathbf a) t^{d(R)-p_j}+\text{(higher terms)},\\
f(\mathbf z(t))&=&b t^{d'}+\text{(higher terms)}, \, d'\ge d(R).
\end{eqnarray}
where 
 $R=(r_1,\dots, r_n)$, $\mathbf a=(a_1,\dots,a_n)\in \mathbb C^{*n}$.  We use the same notation as in \cite{Okabook}.
  If $\De(R)$ is a vertex,  $[R]\in \mathcal D$ and   it is clear that $\frac{\partial f_R}{\partial z_j}(\mathbf a)\ne 0$ for any $j\in {\Var}(R)$
  and $\theta(\mathbf z(t))\le 1-\frac 1L$ by Lemma 10.
 
 Now we assume that $\dim\,\De(R)\ge 1$. 
 Assume that $\frac{\partial f_P}{\partial z_j}(\mathbf a)\ne 0$ for some $j\in \widetilde {\Var}(P)$, it is clear 
 from the definition that $\theta(C(t))\le 1-\hat p_j$. Thus 
  $\theta(\mathbf z(t))\le \theta(R)'$ by the non-degeneracy and the strong inv-tameness. Thus the proof is reduced  to the following assertion.

\begin{Assertion} Assume that $\dim\, \De(R)\ge 1$.
Then 

\[
\theta(R)'\le \max\{\theta (P)'\,|\, P\succ R,\,\dim\,\De(P)=n-1, \, P\in \mathcal W_+\cup\mathcal W_{v}\}.
\]
\end{Assertion}
This is proved easily on the  induction on $\dim\, [R]$, using Lemma 15 and Lemma 16.
If $\dim\,[R]=1$,  $R$ is a vertex and there are nothing to be proved.
Take an admissible line segment ${\LS}(P,Q)$, $P_s=(1-s)P+s Q$ and $R=P_{s_0}$ for some $0<s_0<1$.
By Lemma 15 and Lemma 16,   
we have
the estimation
\[\theta(R)'\le \max\{\theta(P)',\theta(Q)'\}
\]
for $P,Q\in \mathcal W_+\cup\mathcal W_v$ and 
the assertion holds for $R$ by the strong local inv-tameness. The assertion holds by the induction's hypothesis
 for  $P$ and $Q$ if $P,Q\in \mathcal W_+\cup\mathcal W_{v}$.
If $P$ is non-vanishing, the estimation is simply replaced by $\theta(R)'\le \theta(Q)'$.
\end{proof}

\subsection{Examples of the estimation of $\theta_0(f)$}
\begin{Example} {\rm  Consider the polynomial $f_1(\bfz)=z_1^4z_2^2+z_1^6z_3+z_2^6z_3^2+z_1^3z_3^6$ considered in Example 5.
 We have
\[
\theta (P)'=\frac{10}{11},\,\theta (Q)'=\frac 89,\,\theta(R)'=\frac 12,\,\theta(S)'=\frac 45, \, \theta(T)'=\frac 56
\]
The region $A,B,C,D$ corresponds to the monomials  $z_1^5z_2^2,\, z_1^6z_3,\, z_2^6z_3^2,\,z_3^6 z_1^3$ respectively  and these region give the bounds $6/7,\,6/7,\, 7/8,\,8/9$ respectively.
Thus $\theta_0(f_1)\le \dfrac{10}{11}$ by Theorem 22.
}
\end{Example}
\begin{Remark}
{\rm 
The estimation by  Main Theorem 22 is not always sharp. In fact,  the equality in the above estimation can not be obtained. 
For  the weight vector $P$, $f_{1P}(\mathbf z)=z_1^6z_3+z_1^4 z_2^2+z_1^3z_3^6$ and we see that $\frac{\partial f_{1P}}{\partial z_2}\ne 0$ on $\mathbb C^{*3}$. As $\hat P=(\frac 5{33},\frac 3{22},\frac 1{11})$, the real contribution for $P$ is from
 $\frac{\partial f_{P}}{\partial y}$.
Thus $\theta(\mathbf z(t))\le 1-3/22=19/22$ for any $\mathbf z(t)$ with ${\wt}(\mathbf z(t))=P$.
The contribution from $Q$ is in fact sharp.  Note that $f_{1Q}$ is given by $z_1^5z_2^2+z_1^6z_3+z_2^6z_3^2$
and one can find $\mathbf z(t)$ with the coefficient vector $\mathbf a$ satisfies 
$\frac{\partial f_{1Q}}{\partial z_1}(\mathbf a)=\frac{\partial f_{1Q}}{\partial z_2}(\mathbf a)=0$. Thus $\theta(\mathbf z(t))=\frac 89$.
As $\frac 89>\frac {19}{22}$, we conclude $\theta_0(f_1)=\frac 89$.
}
\end{Remark}
\subsection{ Is $\theta_0(f)$ a moduli invariant?}
For a given non-degenerate function $f(\mathbf z)$, we ask if the \L ojasiwwicz exponents are constant or not on the moduli space. For this purpose, we consider 
the branched poly-cyclic  covering $\vphi_2:\mathbb C^n\to \mathbb C^n$
and its lift of $f$, defined by $\vphi_2(\mathbf w)=\mathbf z,\, z_i=w_i ^2\,(1\le  i\le n)$ and put $f^{(2)}(\mathbf w):=\vphi^* f(\mathbf w)=f(w_1^2,\dots, w_n^2)$. More precisely 
we consider $f_1(\mathbf z)=z_1^5 z_2^2+z_1^6 z_3+z_2^6 z_3^2+z_3^6 z_1^3$ in Example 1.
Put $f_1^{(2)}(\mathbf w):=f_1(\vphi_2(\mathbf w))=w_1^{10}w_2^4+w_1^{12} w_3^2+w_2^{12} w_3^4+w_3^{12} z_1^6$.
The dual Newton diagram $\Ga^*(f^{(2)})$ is given by the same diagram of $\Ga^*(f_1)$. Only change is that $d(K, f_1^{(2)})=2d(K, f_1)$
for any weight vector $K$. Thus in the normalized vectors,
$\hat P,  \hat Q, \hat R, \hat S, \hat T$ are to be divided by 2. By the same discussion as in the above Remark, we see that $\theta_0(f_1^{(2)})=\frac{17}{18}$.
Let $g(\mathbf w):=f_1^{(2)}(\mathbf w)+w_1^3 w_2^6 w_3^8$. Note that the new monomial $w_1^3 w_2^6 w_3^8$ corresponds to  the midpoint of the edge ${C^{2)}D^{(2)}}$. 
Here $C^{(2)}, D^{(2)}$ are the lift of $C,D$ in  $\Ga(f^{(2)})$. Thus the dual Newton diagram of $g$ is the same  with $\Ga^*(f^{(2)})$.
By the result of \cite{EO}, the family $g_t(\mathbf w):=f_1^{(2)}(\mathbf w)+tw_1^3 w_2^6 w_3^8$  is non-degenerate and strongly  locally inv-tame except a finite exceptional $t$'s. 
The exceptional set $S$ is the union of $\{\pm 2\}$ from the non-degeneracy of $g_{t,T}$ and possibly  some more $t$'s from the non-degeneracy of $g_{t\hat P}$. Actually $f_{t\hat P}=0$ is non-singular in $\mathbb C^{*3}$ for $t\ne \pm 2$ as we can see by a direct calculation, $\frac{\partial g_{t\hat P}}{\partial z_1}=\frac{\partial g_{t\hat P}}{\partial z_2}=\frac{\partial g_{t\hat P}}{\partial z_3}=0$ has no solution in $\mathbb C^{*3}$. Thus $S=\{\pm 2\}$.
This family
has a canonical  Whitney regular stratification and $V(f^{(2)})$ and $V(g)$ are topologically equivalent for any $t\in\mathbb C\setminus S$. 
We assert $\theta_0(g)=\frac{ 21}{22}\, (=1-\frac 6{132})$ 
which comes from the vertex $P=(10,9,6)$ with $d(P, g)=132.$ 
Thus $\theta_0(g)>\theta_0(f_1^{(2)})$ and $\theta_0$ is not constant on the moduli space of $g$.
Here we mean by moduli the space of functions with fixed Newton boundary and local tameness.
For example, we can choose the following curve which satisfies 
$\frac{\partial g_{P}}{\partial w_1}=\frac{\partial g_P}{\partial w_2}=0$:
\[
w_1(s)=\frac 12 \sqrt[6]{3}\sqrt 2\sqrt[6]{-1}\,s^{10},\, w_2(s)=\sqrt[6]{-\frac 18 i\sqrt{6}}\,s^9,\,w_3(s)=s^6.
\]
We can see that 
\[\begin{split}
&{\ord}_s\partial g(\mathbf w(s))=126,\quad {\ord}_s g(\mathbf w(s))=132\\
&\theta(\mathbf w(s))=\frac{126}{132}=\frac{21}{22}.
\end{split}
\]
Unfortunately $g_t$ has $1$-dimensional singularity. 

For isolated singularity case, this does not happen.
In fact, Brzostowski proved the \L ojasiewicz exponent $\eta_0(f)$ of the \L ojasiewicz inequality of type (2) is constant on the moduli space of functions
with fixed Newton boundary and an isolated singularity at the origin (Theorem 1, \cite{Br2}).
On the other hand, $\theta_0(f)$ and $\eta_0(f)$ are related by 
$\theta_0(f)=\eta_0(f)/(1+\eta_0(f))$ by Teissier \cite{Te}.
 I thank  Professor Tadeusz Krasi\'nski for this information.
\begin{Example}{\rm Consider the simplicial weighted homogeneous polynomial 
$f_2(\mathbf z)=z_1^a  z_2^{2}+z_2^b z_3^{2} +z_3^c z_1^{2}$ in Example 2. Then normalized weight is given as 
$\hat P=(\frac{4-2c+bc}{abc+8}, \frac{ac+4-2a}{abc+8},\frac{ab+4-2b}{abc+8})$. Suppose $c\ge a, b$. Then
the contribution from $\hat P$ is $1-\frac{ab-2b+4}{abc+8}$. The contribution from $\hat R,\hat S,\hat T$, which are $\theta(\hat R)',
\theta(\hat S)',\theta(\hat T)')$,  are  given 
by $1-1/c,1-1/b, 1-1/a$ respectively by Theorem 22 but these estimation is not sharp.
For example, $f_{\hat R}=z_1^az_2^2+z_3^cz_1^2$ and $\frac{\partial f_{\hat R}}{\partial z_2}$ can not be zero on $\mathbb C^{*3}$.
Thus the real contribution is $1-1/2=1/2$. The same is true for $\hat S,\hat T$. Thus 
 $\theta_0(f)=1-\frac{ab-2b+4}{abc+8}$. }
\end{Example}
\begin{Example}
{\rm Consider $f_3(\mathbf z)=z_1^4z_2^{2}+z_2^4z_3^{2}+z_3^4z_1^{2}+z_1z_2z_3$ of Example 3 (See Figure 4).
This polynomial has no compact face of dimension 2 in $\Ga(f)$.
We observe that $\theta(S_1)',\theta(T_1)',\theta(R_1)'=\frac 12$ and 
$\theta(S_2)',\theta(T_2)',\theta(R_2)'=\frac 34$. $\mathcal D$ contains 4 regions.
 The pentagon  with vertices $S_2, S_1, T_2, T_1, R_2, R_1$ contribute by $\frac 23$.
 The other triangles contribute by $\frac 56$. Thus we have $\theta_0(f_3)\le \frac 56$.
 In fact, $\theta_0(f_3)=\frac 56$. To see this, consider the triangle region $S_1T_2E_3$
 in Figure 4 and  take an analytic curve, for example,
 $\mathbf z(t)=(t,t,t^N) $ for a sufficiently large. Then the weight vector is given by
 $P=(1,1,N)$ or $\hat P=(\frac 16,\frac 16, \frac N6)$ and $f_{3P}=z_1^4z_2^2$ and
 $\theta(\mathbf z(t))=\frac 56$.
 
}
\end{Example}
\subsection{$\theta_0(f)$ does not behave like Milnor numbers}
We give another example of a delicate behavior of $\theta_0(f)$.
Assume that $f(\mathbf z),g(\mathbf z)$ have isolated singularities at the origin and they are non-degenerate.
Let $\Ga_-(f),\,\Ga_-(g)$ be the cones of $\Ga(f),\Ga(g)$ with the origin. Assume that $\Ga_-(g)\supsetneq \Ga_-(f)$ and 
$\Ga(f)\cap\mathbb R^I=\Ga(g)\cap \mathbb R^I$ for any $I\subsetneq \{1,\dots,n\}$.
Then by Kouchnirenko's theorem (\cite{Ko}), the Milnor numbers satisfies the inequality: $\mu(g)>\mu(f)$. This is not always true for $\theta_0(g)$ and $\theta_0(f)$.

As an example, consider $g_4(\mathbf z)=z_1^9z_2+z_2^{10}z_3+z_3^{11}z_1$ and $f_4(\mathbf z)=g_4(\mathbf z)+z_1^2z_2^2z_3^2$.
Note that $\Ga_-(g_4)\supset\neq \Ga_-(f_4)$.
First, their Milnor numbers are given as
$\mu(g_4)=990$ and $\mu(f_4)=543$. Their dual Newton boundaries are given as Figure 5 and Figure 6.
\L ojasiewicz exponent is given as
$\theta_0(g_4)=\frac{910}{991}=0.91\cdots$.
On the other hand, $\theta_0(f_4)\le \frac{95}{101}=0.94\cdots$ which comes from $T_3=(\frac{35}{101},\frac{19}{202},\frac{6}{101})$.
In fact, we can show that  the equality $\theta_0(f)=\frac{95}{101}$ is taken by the following curve:
\[
z_1(t)=b_1t^{70}, \,z_2(t)=b_2t^{19},\, z_3(t)=b_3t^{12}
\]
where $\mathbf b\in \mathbb C^{3*}$ satisfies the equality
\[\begin{split}
&\frac{\partial f_{T_3}}{\partial z_1}(\mathbf b)=\frac{\partial f_{T_3}}{\partial z_2}(\mathbf b)=0\\
&f_{T_3}(\mathbf z)=z_2^{10}z_3+z_3^{11}z_1+z_1^2z_2^2z_3^2.
\end{split}
\]
Then we see that ${\ord}_t\partial f_4(\mathbf z(t))=190,\,{\ord}_t f_4(\mathbf z(t))=202$.
For example, we can take
\[
b_1=\root 6\of{\frac{5}{16}},\,b2:=\root {12}\of{\frac 1{20}},\, b_3=-1.
\]
Thus we have the inequality: $\theta_0(f_4)>\theta_0(g_4)$, while $\mu(f_4)<\mu(g_4)$.

\begin{figure}[htb]
\setlength{\unitlength}{1bp}
\begin{picture} (600,300) (-120,-80)
{\includegraphics[width=6cm,  bb=0 0 595 842]{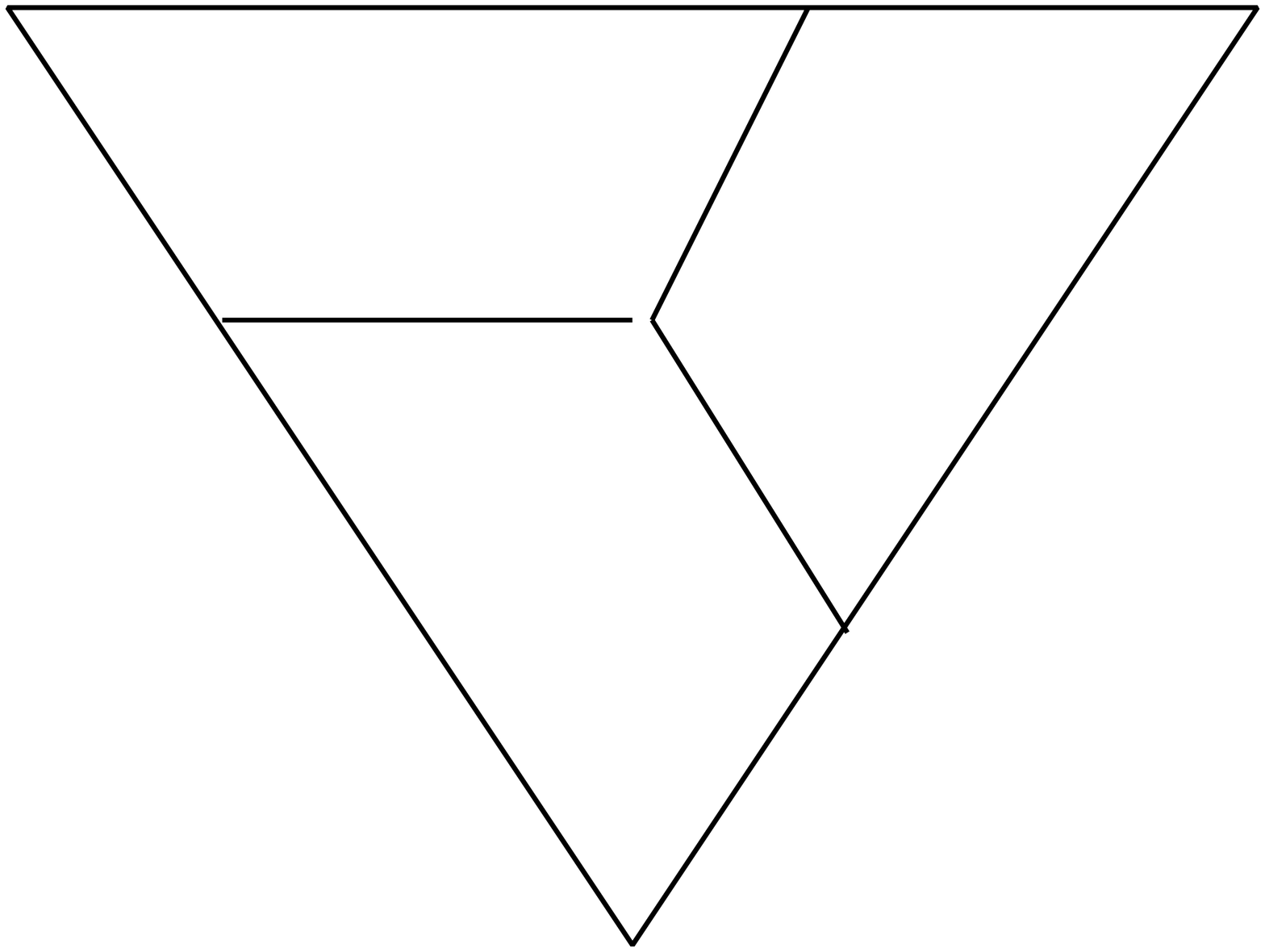}}

\put(5,175){$E_1$} 
\put(-5,180){$\bullet$}
\put(-180,175){$E_2$}
\put(-170,180){$\bullet$}
\put(-100,55){$E_3$}
\put(-88,55){$\bullet$}
\put(-77,186){$S_3$}
\put(-65,180){$\bullet$}
\put(-85,150){$P$}
\put(-85,139){$\bullet$}
\put(-155,140){$S_1$}
\put(-142,139){$\bullet$}
\put(-50,102){$S_2$}
\put(-58,98){$\bullet$}
\put(-280,100){$\hat P=(\frac {100}{991},\frac {91}{991},\frac {81}{991})$}
\put(-280,85){$\hat S_1=(0,1\frac 1{11}),$}
\put(-210,85){$\hat S_2=(\frac19,0,1)$}
\put(-280,70){$\hat S_3=(1,\frac 1{10},0)$}
\put(-280,50){$g_4=z_1^9z_2 +z_2^{10}z_3 +z_3^{11}z_1 $}
\end{picture}
\vspace{-4cm}
\caption{$\Ga^*(g_4)$}\label{DNg}
\label{Graphg}
\end{figure}
\begin{figure}[htb]
\setlength{\unitlength}{1bp}
\begin{picture}(600, 300)(-100,-20)   
{\includegraphics[width=8cm,  bb=0 0 595 842]{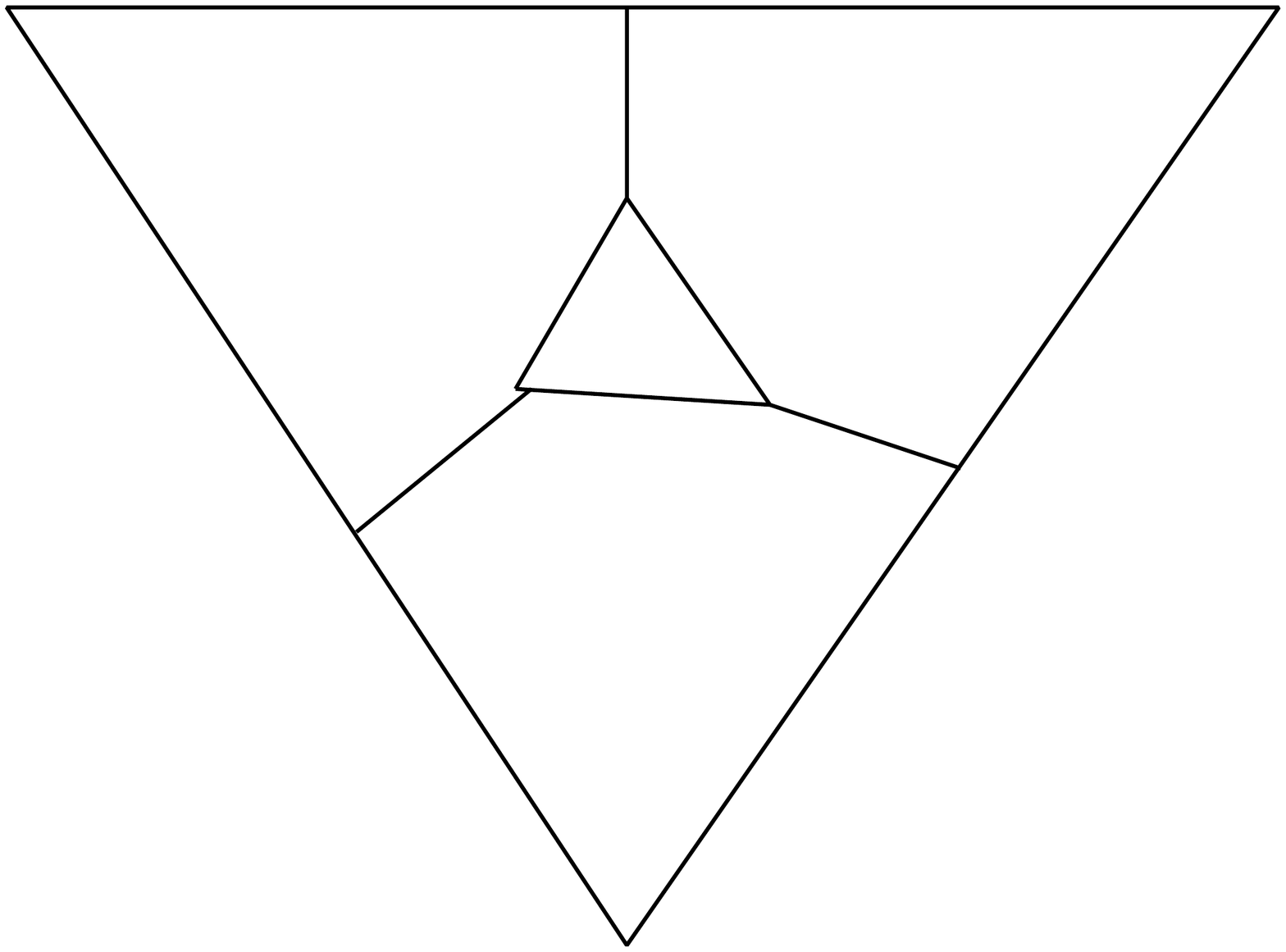}}
\put(5,240){$E_1$}
\put(-5,240){$\bullet$}
\put(-250,240){$E_2$}
\put(-226,240){$\bullet$}
\put(-110,70){$E_3$}
\put(-118,76){$\bullet$}
\put(-114,254){$S_3$}
\put(-118,241){$\bullet$}
\put(-135,163){$T_1$}
\put(-137,173){$\bullet$}
\put(-95,163){$T_2$}
\put(-93,172){$\bullet$}
\put(-110,205){$T_3$}
\put(-118,205){$\bullet$}
\put(-190,145){$S_1$}
\put(-165,148){$\bullet$}
\put(-50,155){$S_2$}
\put(-59,160){$\bullet$}
\put(-320,100){$\hat P=(\frac {100}{991},\frac {91}{991},\frac {81}{991})$}
\put(-320,85){$\hat S_1=(0,1\frac 1{11}),$}
\put(-250,85){$\hat S_2=(\frac19,0,1)$}
\put(-320,70){$\hat S_3=(1,\frac 1{10},0)$}
\put(-320,50){$\hat T_1=(\frac{13}{178},\frac{61}{178},\frac{15}{178}),\hat T_2=(\frac{17}{164},\frac{11}{164},\frac{27}{82}), \hat T_3=(\frac{35}{101},\frac{19}{202},\frac 6{101})$}
\put(-320,30){$g_4=z_1^9z_2 +z_2^{10}z_3 +z_3^{11}z_1 +z_1^2z_2^2z_3^2$}

\end{picture}
\vspace{-2.0cm}
\caption{$f_4(\mathbf z)$}\label{DNf}
\label{Graphf}
\end{figure}

\newpage
\section{\L ojasiewicz exponents of  non-irreducible functions}
In this section, we study \L ojasiewicz exponents of reducible functions which are associated with non-degenerate complete intersection varieties.
\subsection{A function associated with a convenient non-degenerate complete intersection variety}
We consider a family of functions $\mathcal F:=\{f_1,\dots, f_k\}$. We say $\mathcal F$ is
 {\em a  defining family of convenient non-degenerate complete intersection varieties} 
if each $f_\al$ is a convenient non-degenerate function such that 
for  each $I\subset \{1,\dots, k\}$, $V(I):=\{\mathbf z\in \mathbb C^n\,|\, f_i(\mathbf z)=0,\,i\in I\}$ is a non-degenerate complete intersection variety
in the sense of Khovanskii (\cite {Kh2, Okabook}).
 Namely for any strictly positive weight vector $P$,
the variety $V_I(P)^*:=\{\mathbf z\in \mathbb C^{*n}|f_{jP}(\mathbf z)=0\,|\, j\in I\}$ is a smooth complete intersection variety in $\mathbb C^{*n}$.
We consider the product function
$f=f_1\dots f_k$  and  we call  $f$ as {\em the  function associated  with the family $\mathcal F$}. Note that $f$ is also a convenient function.
In \cite{EO14}, we considered a similar family, which we called  {\em a Newton-admissible  family} $\mathcal F=\{f_{1s},\dots, f_{ks}\}$  with one parameter $s\in \mathbb C$ and 
the associated  product function $f_s=f_{1s}\dots f_{ks}$ in $s\in \mathbb C$ to study the topological stability of the one-parameter  family of hypersurfaces $V_s=f_s\inv(0)$. 

In this paper,  we assume that each function $f_j$ is convenient  and has no parameter $s$. However the argument below can be done  in the exact same way  for one parameter family  if their Newton boundary $\Ga(f_{js})$ is independent of $s$. 
As in \S 3.1, we put $b_{\al,j},\ 1\le \al\le k,\,1\le j\le n$ be the unique point of $\Ga(f_\al)$  on $z_j$-axis.
So $z_j^{b_\al,j}$ is a monomial of $f_\al(\mathbf z)$ with non-zero coefficient.
We will give  an explicit estimation  for $\theta_0(f)$.
The hypersurface $V(f)$ has $k$ irreducible components $V(f_\al),\,\al=1,\dots, k$ and $V(f)$ has non-isolated singularities if $k\ge 2$. The   singular locus of $V(f)$ is the union $\cup_{i\ne j}V(f_i,f_j)$. Here $V(f_i,f_j):=\{f_i=f_j=0\}$.

We are interested in the estimation of the \L ojasiewicz  component of $f$.
As in the previous section, we  start to consider an analytic curve $\mathbf z(t),\,0\le t\le 1$ which starts from the origin and
$\mathbf z(t)\in \mathbb C^{*n}\setminus V(f)$ for $t\ne 0$ and  we consider  its Taylor expansion:
\begin{eqnarray}\label{eq11}
&z_j(t)&=a_jt^{p_j}+\text{(higher terms)}, \quad p_j>0,  \, a_j\ne 0,\, j=1,\dots,n.
\end{eqnarray}
(Actually we  also consider the case $\mathbf z(t)\in \mathbb C^{*J}\, (t\ne 0),\, J\subset \{1,\dots,k\}$ such that $f^J$ s not constantly zero. However  the estimation is reduced to the case $\mathbb C^{*n}$ by Proposition 8.)
Put $P=(p_1,\dots,p_n)$ and $\mathbf a=(a_1,\dots, a_n)$ as before.
Let $z_j^{b_{\al j}}$ be the monomial of $\Ga(f_\al)$ on the $j$-th axis for $j=1,\dots, n$
and $1\le \al\le k$.
Let 
\begin{eqnarray}
f_\al(\mathbf z(t))&=&c_\al t^{d_\al}+\text{(higher terms)},\,c_\al\ne 0,\,1\le\al\le k,\label{eq12}\\
 f(\mathbf z(t))&=&c  t^{  d }+\text{(higher terms)},\, c\ne 0,\label{eq13}\\
 d_\al&\ge& d(P,f_\al),\quad   d\ge d(P,f)=\sum_{\al}d(P,f_\al).\label{eq14}
 \end{eqnarray}
 We use this expansion assumption throughout this paper.
Note that  $  d =\sum_{\al=1}^k d_\al$ and $c=c_1\cdots c_k$. 
Put $e_\al:=d(P,f_\al)$ and 
$e =\sum_{\al=1}^k e_\al$.
We observe that 
\begin{eqnarray}\label{eq15}
\frac{\partial f_\al}{\partial z_j}(\mathbf z(t))&=&\frac{\partial f_{\al P}}{\partial z_j}(\mathbf a) t^{e_\al-p_j}+\text{(higher terms)}\notag\\
\partial f(\mathbf z(t))&=&\sum_{\al=1}^k\left( \prod_{\be\ne \al}f_\be(\mathbf z(t))\right)\, \partial f_\al(\mathbf z(t)),\notag
\end{eqnarray}
Therefore by (12) and (13),
\begin{eqnarray}
\frac{\partial f}{\partial z_j}(\mathbf z(t))&=&\sum_{\al=1}^k \left(\prod_{\be\ne \al}f_\be(\mathbf z(t))\right)\,  \frac{\partial f_{\al}}{\partial z_j}(\mathbf z(t))\\
&=&\sum_{\al=1}^k\left\{
\frac{\partial f_{\al P}}{\partial  z_j}(\mathbf a)\frac{c }{c_\al} t^{  d -d_\al+e_\al-p_j}+\text{(higher terms)}
\right\}\notag
\end{eqnarray}
Put $B_j:=\sum_{\al=1}^k b_{\al,j}$ and $B=\max\{B_j\,|\, j=1,\dots, n\}$. Note that 
$z_j^{B_j}$ has  a non-zero coefficient in the expansion of $f(\mathbf z)$ and it corresponds to the unique point of $\Ga(f)$ on the $z_j$-axis.
Let $p_{\min}:=\min\{p_j\,|\, j=1,\dots, n\}$,  $I_{min}=\{j\,|\, p_j=p_{\min}\}$. 
\subsubsection{Case 1.  $\mathbf a$ is generic}\label{generic case}
We consider the case $\mathbf a$ is generic so that $f_P(\mathbf a)=\prod_{\al=1}^k f_{\al P}(\mathbf a)\ne 0$. 
This implies that  $d_\al=e_\al$ for any $\al=1,\dots, k$. As $f_P(\mathbf z)$ is a weighted homogeneous polynomial of degree $ e =\sum_{\al}e_\al$ with respect to the weight vector $P$, $0$ is the only possible critical value of $f_P$. Thus $  d =e=\sum_{\al=1}^k e_\al$ and 
$\partial f_P(\mathbf a)\ne \mathbf 0$ and 
\[
\frac{{{{\ord}_t}}\, \partial f(\mathbf z(t))}{{{{\ord}_t}}\,f(\mathbf z(t))}\le \frac{  e -p_{ min}}{  e }=1-\frac{p_{min}}{  e }.\]
We define a rational number $b_{\al,j}'$ by $b_{\al,j}'p_j=e_\al$ for $j\in I_{min}$.
 Then $p_j B_j'=e $ where $B_j'=\sum_{\al=1}^k b_{\al,j}'$ and  $j_0$ is fixed  in $I_{\min}$. Thus $b_{\al,_{j_0}}'\le b_{\al,j_0}$.
Note  the equality $p_{ j_0}\sum_{\al=1}^k b_{\al,  j_0}'=p_{j_0} B_{j_0}'=  e $.
Thus the above estimation implies 
\begin{eqnarray}\label{hat-theta}
\frac{{{{\ord}_t}}\, \partial f(\mathbf z(t))}{{{{\ord}_t}}\,f(\mathbf z(t))}
\le 1-\frac 1{B_{j_0}'}\le 1-\frac 1{B_{j_0}}\le 1-\frac 1 B.
\end{eqnarray}
\subsubsection{Case 2. $\mathbf a$ is not generic}
This case is  more complicated.
We consider non-generic coefficient vector $\mathbf a$. So we assume that there exists   $\al,\,1\le\al\le k$ such that  $f_{\al P}(\mathbf a)=0$.
Consider the defect numbers $d_j':=d_j-e_j,\, 1\le j\le k$ and changing the numbering of $f_j,\,1\le j\le k$ if necessary, we
assume for simplicity 
\begin{eqnarray}\label{eq17}
d_1'\le d_2'\le\dots \le d_{\ell-1}'< d_\ell'=\dots=d_k'\, \, 
\end{eqnarray}
for some $\ell,\,1\le \ell\le k$. Therefore  $\ell:=\min\,\{\al\,|\, d_\al=d_k'\}$.
Note  that $f_{\al P}(\mathbf a)\ne 0$ if and only if  $d_\al'=0$. In particular 
$f_{\al P}(\mathbf a)=0$  as $d_\al'>0$ for $\ell\le \al\le k$.  We have  the estimation:
\begin{eqnarray*}
&{{{{\ord}_t}}}\,&\left\{\left( \prod_{\be\ne \al}{f_\be(\mathbf z(t))}\right)\frac{\partial f_\al}{\partial z_j}(\mathbf z(t))\right\}\\
&\quad \ge&   d -d_\al +(e_\al-p_j)= d-d_\al'-p_j,\,  1\le  j\le n.
\end{eqnarray*}
Note that 
\begin{eqnarray}\label{eq18} 
d-d_1'-p_j&\ge&\dots\ge d-d_{\ell-1}'-p_j\\
&>&d-d_\ell'-p_j=\cdots=d-d_k'-p_j.\notag
\end{eqnarray}
and  finally we have the expression:
\begin{eqnarray}\label{non-zero}
\frac{\partial f}{\partial z_j}(\mathbf z(t))=\left(\sum_{\al=\ell}^k \frac c{c_\al}\frac{\partial f_{\al P}}{\partial z_j}(\mathbf a)\right)
t^{d-d_\ell'-p_j}+\text{(higher terms)}.
\end{eqnarray}
\begin{Assertion}\label{lin-ind}
There exists some $j_0$ such that 
${\ord}_t\,\frac{\partial f}{\partial z_{j_0}}(\mathbf z(t))=d-d_\ell'-p_{j_0}$.
\end{Assertion}
\begin{proof}
By the assumption,  we have $d_\ell'=\cdots=d_k'$. For the proof of the assertion, 
we use   the non-degeneracy assumption of the complete intersection variety
$V_I(P)^*:=\{\mathbf z\in \mathbb C^{*n}\,|\, f_{\al P}(\mathbf z)=0,\, \al\in I\}$ with $I=\{\ell,\dots,k\}$.  By the assumption,  $\mathbf a\in V_I(P)^*$.
Assume that $\sum_{\al=\ell}^k \frac c{c_\al}\frac{\partial f_{\al P}}{\partial z_j}(\mathbf a)=0$ for any $j$.
This implies
$\sum_{\al=\ell}^k \frac c{c_\al} \partial f_{\al P}(\mathbf a)=0$ and it gives a non-trivial linear relation among  gradient vectors
$\partial f_{\ell P}(\mathbf a),\dots, \partial f_{k P}(\mathbf a)$ which is a contradiction to
the non-degeneracy  of the complete intersection assumption $V_I(P)^*$.
\end{proof} 
Thus 
there exists $j_0,\,1\le j_0\le n$ so that
\[
\sum_{\al=\ell}^k \frac c{c_\al}\frac{\partial f_{\al P}}{\partial z_{j_0}}(\mathbf a)\ne 0,\, \text{that is},\,\,
{\ord}_t\frac{\partial f}{\partial z_{j_0}}(\mathbf z(t)) =d-d_\ell'-p_{j_0}.
\] 
Consider the integers: 
 \[D_{\al-1}':=d_1'+\dots+d_{\al-1}',\,1\le \al\le k . 
\]
We have
\[\begin{split}
{\ord}_t f(\mathbf z(t))&=d_1+\dots+d_k
=D_{\ell-1}'+(k-\ell)d_\ell'+ e,\\
{\ord}_t\,\partial f(\mathbf z(t))
&=\inf\{ {{\ord}_t}\, \frac{\partial f}{\partial z_j}(\mathbf z(t))\,|\, 1\le j\le n\}\\
&\le
\sup\{ d -d_\ell'-p_j\,|\, 1\le j\le n\}\\
&\le D_{\ell-1}'+(k-\ell-1){d_\ell}'+ e-\frac {   e }{ B} 
\end{split}
\]
Here we have used the equality $p_j\ge e/B_j\ge e/B$ and 
$d_\al=d_\al'+e_\al$,  in particular $d_\al=d_\ell'+e_\al$ for $\al\ge \ell$. Thus we have
\begin{eqnarray}
&\theta(\mathbf z(t))=\dfrac{{\ord}_t\,\partial f (z(t))}{{\ord}_t\,f(\mathbf z(t))}\le F_\ell
\end{eqnarray}
where $F_\ell$ is defined by the following:
\[\begin{split}
F_\ell:=&\frac{D_{\ell-1}'+(k-\ell-1)d_\ell'+e-  e / B}{D_{\ell-1}'+(k-\ell )d_\ell'+ e}.
\end{split}
\]
Note that under the assumption that $\mathbf z(t)$ has the  weight vector $P$ and (17), $d_1',\dots, d_\ell'$ are not constant but the other numbers are constant. Originally $d_\ell'$ is an integer but we extend to real numbers so that 
$F_\ell$ is a function of $d_\ell'$ on the interval $[d_{\ell-1}',\infty)$, fixing  $d_1',\dots, d_{\ell-1}'$.
Note that $F_0=1-1/B$.

\subsection{Comparison with $F_{\ell-1}$}
We want to  show $F_\ell\le 1- 1/B$.
We assert that 
 $F_\ell$  is monotone decreasing function of $d_\ell'$, fixing $d_1',\dots, d_{\ell-1}'$ where $d_j':=d_j-e_j$.
 Here $d_\ell'$ moves on  the interval $[d_{\ell-1}',\infty)$.
In fact, the differential of the right hand side in $d_\ell'$ is given as 
\[
\frac{\partial F_\ell}{\partial d_\ell'}=\frac{-D_{\ell-1}'- e+(k-\ell)   e /B}{(D_{\ell-1}'+(k-\ell)d_\ell'+ e)^2}. 
\]
We assert that 
\begin{Lemma} $F_\ell$ is monotone decreasing function of $d_\ell'$ as 
\[
\frac{\partial F_\ell}{\partial d_\ell'}\le 0.
\]
\end{Lemma}
\begin{proof}
The numerator of the differential $ {\partial F_\ell}/{\partial d_\ell'}  $ can be estimated as 
\[\begin{split}
-D_{\ell-1}'- e+(k-\ell) e /B&=-D_{\ell-1}'-  e (1-\frac{k-\ell}B)\\
&\le 0.
\end{split}
\]
Here we have used the obvious inequality $B\ge k$.
\end{proof}
Thus putting $d_\ell'=d_{\ell-1}'$, we  get the estimation
$\theta_0(f)\le F_\ell\le F_{\ell-1}$
where $F_{\ell-1}$ is obtained by substituting $d_\ell'=d_{\ell-1}'$:
\[
F_{\ell-1}:=F_\ell|_{d_\ell'=d_{\ell-1}'}=\frac{D_{\ell-2}'+(k-\ell)d_{\ell-1}'+e-  e / B}{D_{\ell-2}'+(k-\ell +1)d_{\ell-1}'+ e}
\]
where $D_{\ell-2}'=d_1'+\dots+d_{\ell-2}',\, e=e_{1}+\dots+e_k$. 
\subsection{ \L ojasiewicz exponents for  the product functions}
Now we  are ready to state our main result for the product function.
\begin{MainTheorem}\label{Product-Lo}
 Let $f=f_1\cdots f_k$ be the product function associated to a generating family of  convenient non-degenerate 
complete intersection varieties $\mathcal F=\{f_1,\dots, f_k\}$. Then the \L ojasiewicz exponent of type (1) satisfies the   inequality:
$ \theta_0(f)\le 1-1/B$ and the equality holds if $f$  has a \L ojasiewicz non-exceptional  monomial.
\end{MainTheorem}
\begin{proof}
Continuing the above argument repeatedly,  $\theta_0(f)$ can be estimated  by assuming
$d_\ell'=\dots=d_1'$ as follows:
\[
F_\ell\le F_{\ell-1}\le\cdots\le F_1=\frac{(k-1)d_1'+e - {e }/B }{(k-1)d_1'+e }.
\]
 As $F_1$ is also a monotone increasing function of $d_1'$, putting $d_1'=0$,
    we conclude $F_\ell\le F_0=1- 1/B$. This implies
\[
 \theta_0(f)\le  \frac{e - {e }/B}{e }=1-\frac 1B.
\]
For the existence of the curve attaining the equality under the assumption of the existence of \L ojasiewicz non-exceptional monomial,  we do the same argument as in \S 3.1.
\end{proof}

\subsection{Generalization  to non-reduced functions}
We will show  that   \L ojasiewicz exponent of  a non-reduced expression is determined by the reduced one.
First suppose that $f$ is a reduced function and   let $g(\mathbf z)=f^m(\mathbf z)$.
Let $\mathbf z(t),\,0\le t\le 1$ be an analytic curve starting from the origin and $\mathbf z(t)\in \mathbb C^n\setminus V(f)$ for $t\ne 0$ as before. 
Then we  observe that
\begin{eqnarray}\label{non-reduced}
{\ord}_t\frac{\partial g}{\partial z_j}(\mathbf z(t))=(m-1){\ord}_tf(\mathbf z(t))+
{\ord}_t\frac{\partial f}{\partial z_j}(\mathbf z(t)).
\end{eqnarray}
To distinguish two \L ojasiewicz exponents of $f$ and $g$, we put  
\[\begin{split}
\theta_{f}(\mathbf z(t)):=\frac{{\ord}_t\partial f(\mathbf z(t))}{{\ord}_t f(\mathbf z(t))},\,\,
\theta_{g}(\mathbf z(t)):=\frac{{\ord}_t\partial g(\mathbf z(t))}{{\ord}_t g(\mathbf z(t))}.
\end{split}
\]
In particular, we have the equality
\[\begin{split}
\theta_g(\mathbf z(t))&=
\frac{{\ord}_t\frac{\partial g}{\partial z_j}(\mathbf z(t))}{{\ord}_t g(\mathbf z(t))}\\
&= \frac{{(m-1){\ord}_t f(\mathbf z(t))+\ord}_t\frac{\partial f}{\partial z_j}(\mathbf z(t))}{m\,{\ord}_t f(\mathbf z(t))}\\
&=\frac{m-1}m+\frac 1m \theta_f(\mathbf z(t)).
\end{split}
\]
Thus we have 
\begin{Proposition} The \L ojasiewicz exponents  of $f$ and $g=f^m$ are 
 related by the equality:
\[
\theta_0(g)=\frac {m-1}m+\frac 1m \theta_0(f).
\]
\end{Proposition}

This observation can be generalized to our product function $f$. discussed in \S 4.1.
We consider a defining family $ \mathcal F=\{f_1,\dots, f_k\}$
  of convenient non-degenerate complete intersection varieties as \S 4.1.
  Let $f(\mathbf z):=f_1(\mathbf z)\cdots f_k(\mathbf z)$.
  We consider also non-reduced product function 
  \[
  g(\mathbf z) =  f_1^{m_1}(\mathbf z) \cdots f_k^{m_k}(\mathbf z)
  \]
  where  $m_1,\dots, m_k$ are positive integers.
  Let $\mathbf z(t)$ be an analytic curve  expanded as (11), (12) and (13).
  We use the same notations of numbers $e_j,d_j, d_j',  e$.  
    We define new  integers $\widetilde d_j,\,\widetilde e_j,\widetilde d,\widetilde e $ and complex numbers $\widetilde c_\al, \widetilde c, \widetilde d$ 
    as
    \[\begin{split}
    \widetilde  d_\al&=m_\al d_\al,\, \widetilde e_\al=m_\al e_\al, \,\widetilde d=\sum_{\al=1}^k m_\al d_\al,\, \widetilde e=\sum_{\al=1}^k m_\al e_\al\\
    \widetilde c_\al&=c_\al^{m_\al},\, \widetilde c=\prod_{\al=1}^k c_\al^{m_\al}
    \end{split}
    \]
    so that 
   \begin{eqnarray*}
g_\al(\mathbf z(t))&=&f_\al(\mathbf z(t))^{m_\al}=\widetilde c_\al t^{\widetilde d_\al}+\text{(higher terms)},\,1\le\al\le k\\
 g(\mathbf z(t))&=&\widetilde c\,  t^{ \widetilde  d }+\text{(higher terms)}.
 \end{eqnarray*}
 We work under the  same assumption (17):
 \begin{eqnarray}
   d_1'\le   d_{2}' \dots\le   d_{\ell-1}'<  d_\ell'=\cdots=  d_k'. \notag
 \end{eqnarray}
 We proceed by the exact same argument as  the one in the reduced case. 
 The equality (19) is replaced as 
\begin{eqnarray}
 \frac{\partial g}{\partial z_j}(\mathbf z(t))&=& \sum_{\al=1}^k \left(\left(\frac{m_\al }{f_\al}g\right)(\mathbf z(t))\frac{\partial f_\al}{\partial z_j} (\mathbf z(t))\right)\\
 &=&\left(\sum_{\al=\ell}^k \frac{m_\al \widetilde c}{\widetilde c_\al}
 \frac{\partial f_{\al P}}{\partial z_j}(\mathbf a)
 \right)
 t^{\widetilde d- d_\ell'-p_j}+\text{(higher terms)}.\notag
 \end{eqnarray}
 Define $\widetilde D_\al$ in the same manner as in \S 4.1:
 \[
 \widetilde D_\al=\widetilde d_1+\dots+\widetilde d_\al=\sum_{i=1}^\al m_id_i.
 \]
 $\widetilde B_j$ corresponds the point $\Ga(g)\cap\{z_j\text{-axis}\}$
 which is equal to $\sum_{\al=1}^k m_\al b_{\al,j}$ and $\widetilde B$ is the maximum of $\{\widetilde B_1,\dots,\widetilde B_n\}$.
 So $p_{min}\widetilde B\ge \widetilde e$.
 As 
 the gradient  vectors
 \[
 \left\{
 \frac{\partial f_{\al P}}{\partial z_j}(\mathbf a)\,|\, \al=\ell,\dots, k\right\}
 \]
  are linearly independent by the non-degeneracy of the intersection variety $V_I^*(P)$ (see the proof of Assertion 28), we have
 \[\begin{split}
 \ord\,g(\mathbf z(t))&=\widetilde d_1+\dots+\widetilde d_k=\widetilde D_{\ell-1}'+(m_\ell+\dots+m_k)d_\ell'+\widetilde e,\\
 \ord\,\partial g(\mathbf z(t))&\le\widetilde d-d_\ell'-p_{min}\le
 \widetilde D_{\ell-1}'+(m_\ell+\dots+m_k-1)d_\ell'+\widetilde e-\frac{\widetilde e}{\widetilde B}.
 \end{split}
 \]
  Thus
  we can modify  equality  (20) as :
 \begin{eqnarray}
 \theta_g(\mathbf z(t))=\frac{{\ord}_t\partial g(\mathbf z(t))}{{\ord}_t g(\mathbf z(t))}\le \widetilde F_\ell
 \end{eqnarray}
 where
 \[
 \widetilde F_\ell=\frac{ \widetilde D_{\ell-1}'+( m_\ell+\dots+m_k-1) d_\ell'+{\widetilde e} -\widetilde e/\widetilde B  }
 { \widetilde D_{\ell-1}'+(m_\ell+\dots+m_k) d_\ell'+{\widetilde e} }.
 \]
 and we have
 \[
 \frac{\partial \widetilde F_\ell}{\partial d_\ell'}=
 \frac{-\widetilde D_\ell-\widetilde e+(m_\ell+\cdots+m_k)\widetilde e/\widetilde B  }{ (\widetilde D_{\ell-1}'+(m_\ell+\dots+m_k) d_\ell'+{\widetilde e}) ^2}<0
 \]
 where  the negativity is derived from the fact $\widetilde B\ge m_1+\dots+m_k$.
By the exact same argument, we get the generalization of Theorem 30:
 \begin{Theorem} The \L ojasiewicz exponent of $g=f_1^{m_1}\cdots f_k^{m_k}$ can be estimated as 
 \[
 \theta_0(g)\le  
 1-\frac 1{\widetilde B}.
 \]
 Furthermore, if $g$ has a non-exceptional \L ojasiewicz monomial, the equality holds.
 \end{Theorem}
 We comment that $\widetilde B=\max\,\{m_1b_{1j}+\dots+m_k b_{kj}\,|\, j=1,\dots, n\}$.
\def\cprime{$'$} \def\cprime{$'$} \def\cprime{$'$} \def\cprime{$'$}
  \def\cprime{$'$} \def\cprime{$'$} \def\cprime{$'$} \def\cpri{$'$}

\end{document}